\documentclass[a4paper, 11pt]{amsart}
\usepackage{amsmath, amsthm, amssymb, fullpage, graphicx, url, enumerate, color, bookmark, tikz}
\usetikzlibrary{decorations.pathreplacing}
\usepackage{hyperref}

\usepackage{stmaryrd} 
\usepackage{mathtools} 
\newcommand{\accentset}[2]{\overset{\mathclap{{}_{#1}}}{#2}}

\newtheorem{definition}{Definition}
\newtheorem{theorem}{Theorem}
\newtheorem{proposition}[theorem]{Proposition}

\newtheorem{corollary}[theorem]{Corollary}
\newtheorem*{corollary*}{Corollary}
\newtheorem{conjecture}{Conjecture}
\theoremstyle{remark}

\newtheorem*{fact}{Fact}

\newcommand{\latin}[1]{#1}

\newcommand{\R}{\mathbb{R}}
\newcommand{\N}{\mathbb{N}}
\newcommand{\Z}{\mathbb{Z}}
\newcommand{\Lr}{\mathcal L}

\newcommand{\eps}{\varepsilon}

\newcommand*{\PP}{\mathbb{P}}
\newcommand*{\EE}{\mathbb{E}}
\newcommand*{\prb}[1]{\PP(#1)}
\newcommand{\from}{\mathbin{\leftarrow}}
\newcommand{\ffrom}{\mathbin{\accentset{\!1}\leftarrow}} 
\newcommand{\tto}{\mathbin{\rightarrow}}
\newcommand{\fto}{\mathbin{\accentset{\!1}\rightarrow}} 
\newcommand{\collide}{\mathbin{\rightarrow\leftarrow}}

\newcommand{\go}{\accentset{\to}{\bullet}}
\newcommand{\come}{\accentset{\leftarrow}{\bullet}}
\newcommand{\stay}{\dot{\bullet}}

\DeclareMathOperator{\rev}{rev}
\newcommand{\Fr}{\mathcal F}
\newcommand{\indic}{{\bf 1}}

\newcommand{\iinter}[1]{\llbracket #1 \rrbracket}

\newcommand{\defeq}{\mathrel{\mathop:}=}

\newcommand{\st}{\,:\,} 
\newcommand{\s}{\mid} 

\DeclareMathOperator{\sky}{skyline}
\newcommand{\perm}{\mathfrak S}

\newcommand{\Ell}{\mathbb L}
\newcommand{\Ellgen}{\Ell_{\mathrm{generic}}}
\newcommand{\Ellsin}{\Ell_{\mathrm{single}}}
\newcommand{\Ellmul}{\Ell_{\mathrm{multiple}}}

\newcommand{\triright}{\begin{tikzpicture}\draw[rounded corners=0.02ex] (0,0) -- (0,1.5ex) -- (1.5ex,0) -- cycle;\end{tikzpicture}}
\newcommand{\trileft}{\begin{tikzpicture}\draw[rounded corners=0.02ex] (0,0) -- (1.5ex,1.5ex) -- (1.5ex,0) -- cycle;\end{tikzpicture}}
\newcommand{\triboth}{\begin{tikzpicture}\draw[rounded corners=0.02ex] (0,0) -- (0.75ex,1.5ex) -- (1.5ex,0) -- cycle;\end{tikzpicture}}

\title{Combinatorial universality in three-speed ballistic annihilation}

\author[J. Haslegrave]{John Haslegrave}
\address{Mathematics Institute, University of Warwick, Coventry, UK}

\author[L. Tournier]{Laurent Tournier}
\address{Universit\'e Sorbonne Paris Nord, LAGA, CNRS, UMR 7539,  F-93430, Villetaneuse, France.}

\begin{document}

\begin{abstract}
We consider a one-dimensional system of particles, moving at constant velocities chosen independently according to a symmetric distribution on $\{-1,0,+1\}$, and annihilating upon collision -- with, in case of triple collision, a uniformly random choice of survivor among the two moving particles. When the system contains infinitely many particles, whose starting locations are given by a renewal process, a phase transition was proved to happen (see~\cite{HST}) as the density of static particles crosses the value $1/4$. Remarkably, this critical value, along with certain other statistics, was observed not to depend on the distribution of interdistances. In the present paper, we investigate further this universality by proving a stronger statement about a finite system of particles with fixed, but randomly shuffled, interdistances. We give two proofs, one by an induction allowing explicit computations, and one by a more direct comparison. This result entails a new nontrivial independence property that in particular gives access to the density of surviving static particles at time~$t$ in the infinite model. Finally, in the asymmetric case, further similar independence properties are proved to keep holding, including a striking property of gamma distributed interdistances that contrasts with the general behavior.


\end{abstract}

\maketitle

\section{Introduction}

Annihilating particle systems have been studied extensively in statistical physics since the 1980s. The original motivation for this topic, stemming from the kinetics of chemical reactions, gave rise to models in which particles move diffusively and are removed from the system upon meeting another particle (e.g.~\cite{arratia1983site}), or in some settings another particle of a specified type (e.g.~\cite{bramson1991asymptotic}). The study of annihilating systems involving particles which move at constant velocity (that is, ballistic motion) was initiated by Elskens and Frisch~\cite{EF85} and Ben-Naim, Redner and Leyvraz~\cite{ben1993decay}. In this ballistic annihilation process particles start at random positions on the real line and move at randomly-assigned constant velocities, annihilating on collision. This process displays quite different behavior to diffusive systems and its analysis presents particular challenges.

In order to specify a precise model, we must choose how particles are initially positioned on the real line and how velocities are initially assigned to particles. The most natural choice for the initial positions is arguably the points of a homogeneous Poisson point process, which is the only choice considered in the physics literature. It is also natural to sample i.i.d.\ velocities from some distribution. In the case of a discrete distribution supported on two values, it is easy to see that almost surely every particle is eventually destroyed when the two velocities have equal probability, but that almost surely infinitely many particles of the more probable velocity survive otherwise. However, this model still displays interesting global behavior; see e.g.~\cite{belitsky1995ballistic}.

The first case for which the question of survival of individual particles is not trivial is therefore a three-valued discrete distribution. Krapivsky, Redner and Leyvraz \cite{krapivsky1995ballistic} considered the general symmetric distribution on $\{-1,0,+1\}$, i.e.\ $\frac{1-p}2\delta_{-1}+p\delta_0+\frac{1-p}2\delta_{+1}$ for some $p\in(0,1)$. They predicted the existence of a critical value $p_{\mathrm c}$ such that for $p\leq p_{\mathrm c}$ almost surely every particle is eventually destroyed and for $p>p_{\mathrm c}$ almost surely infinitely many particles survive, and further that $p_{\mathrm c}=1/4$. Even the existence of such a critical value is far from obvious, given that there is no coupling to imply monotonicity of the annihilation of particles with $p$.

These predictions were strongly supported by intricate calculations of Droz, Rey, Frachebourg and Piasecki \cite{droz1995ballistic}. More recently, this model attracted significant interest in the mathematics community. Rigorous bounds were established by Sidoravicius and Tournier~\cite{sidoravicius-tournier}, and independently by Dygert et al.~\cite{junge}, giving survival regimes for $p\geq0.33$, but a subcritical regime was more elusive. In previous work with the late Vladas Sidoravicius~\cite{HST}, we established the precise phase transition predicted by Krapivsky, Redner and Leyvraz.

A closely-related problem known as the bullet problem was popularised by Kleber and Wilson~\cite{ibm}. In this problem, a series of bullets with random speeds are fired at intervals from a gun, annihilating on collision. Kleber and Wilson~\cite{ibm} asked for the probability that when $n$ bullets are fired, all are destroyed. Broutin and Marckert~\cite{BM19} solved this problem in generality, by showing that the answer does not depend on the choice of speeds or intervals, provided that these are symmetrical. For any fixed sequence of $n$ speeds and $n-1$ time intervals, consider firing $n$ bullets with a random permutation of the speeds, separated by a random permutation of the intervals. Provided that the speeds and intervals are such that triple collisions cannot occur, they show that the probability that all particles are destroyed, and even the law of the number of particles which are destroyed, does not depend on the precise choice of speeds and intervals. However, they note that this universality property does not extend further, in that the indices of surviving particles does depend on the choice of speeds and intervals.

In proving the phase transition for symmetric three-speed ballistic annihilation, we observed a form of universality applies. Consider a one-sided version where particles are placed on the positive real line. For each $n$, the probability that the $n$th particle is the first to reach $0$ is universal provided that distances between initial positions of particles are i.i.d. This universality extends to discrete distributions, provided that triple collisions are resolved randomly. These probabilities satisfy a recurrence relation which may be leveraged to prove asymptotics for the decay of particles; see \cite{HST} for further details. Note that this universality property encompasses more information than that of Broutin and Marckert, since it relates to the indices of surviving particles, which are not universal in their case; indeed, some further information on the fate of the remaining particles (the ``skyline process'' of \cite{HST}) is universal, although this universality does not extend to the full law of pairings. However, it relies on successive intervals between particles being independent, which is not necessary in \cite{BM19}.

In this article we extend this stronger form of universality for the symmetric three-speed case to the combinatorial setting of Broutin and Marckert. These results are specific to the symmetric three-speed case (that is, where the three speeds are in arithmetic progression and the two extreme speeds have equal probability), and this symmetry is necessary for the quantities we consider to be universal. However, we do prove (see Section~\ref{sec:asymmetric}) some unexpected properties of particular interdistance distributions which extend to the asymmetric three-speed case. This case was considered by Junge and Lyu~\cite{junge-lyu-asymmetric}, who extended the methods of \cite{HST} to give upper and lower bounds on the phase transition.

The so-called ``bullet problem'', before it was put in relation with the topic of ``ballistic annihilation'' in the physics literature, gained considerable visibility in the community of probabilists thanks to Vladas' warm descriptions and enthusiasm. It soon became one of Vladas' favorite open problems that he enjoyed sharing around him, and that he relentlessly kept investigating. The second author had innumerable lively discussions with him on this problem over the years, first already in Rio, shortly before his departure, and then in Shanghai. It is mainly thanks to Vladas' never-failing optimism when facing difficult problems that, after years of vain attempts and slow progress, efforts could be joined to lead to a solution in the discrete setting with three speeds and to Paper~\cite{HST}.

\section{Definitions and statements}

We define ballistic annihilation with either fixed (shuffled) or random (i.i.d.) lengths.

\subsection{Fixed lengths}
Let us first define the model in the combinatorial setting that is specific to this paper. For integers $a\leq b$ we write $\iinter{a,b}$ for the set of integers in $[a,b]$.

Let $n$ be a positive integer, $\mu$ be a symmetric distribution on $\{-1,0,+1\}$, i.e.\ for some $p\in[0,1]$ we have $\mu=\frac{1-p}2\delta_{-1}+p\delta_0+\frac{1-p}2\delta_{+1}$, and $\ell=(\ell_1,\ldots,\ell_n)$ be an ordered $n$-uple of positive real numbers:
\[\ell\in\Ell_n=\big\{(\ell_1,\ldots,\ell_{n})\in(0,+\infty)^n\st \ell_1\le\cdots\le\ell_{n}\big\}.\] 
On a probability space $(\Omega,\Fr,\PP)$, consider independent random variables $(v,s,\sigma)$ such that
\begin{itemize}
	\item $v=(v_1,\ldots,v_n)$ where $v_1,\ldots,v_n$ are independent, with distribution $\mu$; 
	\item $s=(s_1,\ldots,s_n)$ where $s_1,\ldots,s_n$ are independent, uniformly distributed on $\{-1,+1\}$;
	\item $\sigma$ is  uniformly distributed on the symmetric group $\perm_n$. 
\end{itemize}
Finally, define the positive random variables $x_1,\ldots,x_n$ by
\[x_1=\ell_{\sigma(1)},\qquad x_2=\ell_{\sigma(1)}+\ell_{\sigma(2)},\qquad\ldots,\qquad x_n=\ell_{\sigma(1)}+\cdots+\ell_{\sigma(n)}.\]

We interpret $n$ as a number of particles, $x_1,x_2,\ldots,x_n$ as their initial locations, $v_1,v_2,\ldots,v_n$ as their initial velocities, and $s_1,s_2,\ldots,s_n$ as their ``spins''.
For any $i\in\iinter{1,n}$, the spin $s_i$ will only play part in the process if $v_i=0$, in which case it will be used to resolve a potential \emph{triple collision} at $x_i$. 

In notations, the particles will conveniently be referred to as $\bullet_1,\bullet_2,\ldots,\bullet_n$, and particles with velocity $0$ 
will sometimes be called \emph{static} particles, as opposed to \emph{moving} particles. 

The evolution of the process of particles describes as follows (see also Figure~\ref{fig:skyline}): at time~$0$, particles $\bullet_1,\ldots,\bullet_n$ respectively start at $x_1,\ldots,x_n$, then move at constant velocity $v_1,\ldots,v_n$ until, if ever, they collide with another particle. Collisions resolve as follows: where exactly two particles collide, both are annihilated; where three particles, necessarily of different velocities, collide, two are annihilated, and either the right-moving or left-moving particle survives (i.e.\ continues its motion unperturbed), according to the spin of the static particle involved. Note that each spin affects the resolution of at most one triple collision. Annihilated particles are considered removed from the system and do not take part in any later collision. Thus, after a finite time, every particle is either annihilated or shall pursue its ballistic trajectory forever.

For a more formal definition, we refer the interested reader to~\cite{HST}.

\subsection{Random lengths}
We may alternatively, in accordance to the classical setting of ballistic annihilation, consider the distances $\ell_1,\ell_2,\ldots$ to be random, independent and identically distributed (i.i.d.), which naturally enables to extend the definition into an infinite number of particles. 

Let $\mu=\frac{1-p}2\delta_{-1}+p\delta_0+\frac{1-p}2\delta_{+1}$ be a symmetric distribution on $\{-1,0,+1\}$ and let $m$ be a probability measure on $(0,\infty)$. On $(\Omega,\Fr,\PP)$ we consider random variables $(\ell,v,s)$ where
\begin{itemize}
	\item $\ell=(\ell_k)_{k\ge1}$ where $\ell_1,\ell_2,\ldots$ are independent, with distribution $m$;
	\item $v=(v_k)_{k\ge1}$ where $v_1,v_2,\ldots$ are independent, with distribution $\mu$;
	\item $s=(s_k)_{k\ge1}$ where $s_1,s_2,\ldots$ are independent, uniformly distributed on $\{-1,+1\}$. 
\end{itemize}
In contrast to the previous setting, we define, for all $n\ge1$, 
\[x_n=\ell_1+\cdots+\ell_n.\]
In other words, $(x_n)_{n\ge1}$ is a renewal process on $(0,\infty)$ whose interdistances are $m$-distributed. 

The process is then defined in the same way as in the finite case, now with infinitely many particles respectively starting from $x_1,x_2,\ldots$ Note that triple collisions may only happen when the distribution $m$ has atoms, hence in the opposite case the sequence $s$ of spins is irrelevant. Let us nonetheless already mention that some arguments (cf.\ Section~\ref{sec:universality_direct}, most notably) involve the atomic case toward an understanding of the continuous case. 

The model in this infinite setting goes through a phase transition as $p$ varies from $0$ to $1$, specifically at $p=1/4$, from annihilation of all static particles to survival of a positive density of them. This was the main result of~\cite{HST}; it is not used in the present paper, although some remarks refer to it. 

While most of the paper is concerned with the fixed lengths setting, some consequences about random lengths will be mentioned (Corollary~\ref{cor:xA}), along with a specific property of gamma distributions, Theorem~\ref{thm:independence}. Unless otherwise specified, the model under consideration therefore has fixed lengths, hence finite size. 

\subsection{Notation}

We introduce convenient abbreviations, borrowed from~\cite{HST}, to describe events related to the model. 

We use $\bullet_i$ (where $1\le i\le n$) for the $i$th particle, $\bullet$ (with no subscript) for an arbitrary particle, and superscripts $\go$, $\stay$ and $\come$ to indicate that those particles have velocity $+1$, $0$ and $-1$ respectively. 

We write $\bullet_i\sim_\ell\bullet_j$ to indicate mutual annihilation between $\bullet_i$ and $\bullet_j$, which depends on the fixed lengths $\ell=(\ell_1,\ldots,\ell_n)$. Let us indeed emphasize that $\ell$ plays no role in $\PP$, but in the definition of collisions. Still, instead of writing $\ell$ in subscripts, we will often, for readibility, write $\PP_\ell$ to emphasize the choice of $\ell$, and drop the notation $\ell$ from the events. 

Every realization of $(v,s,\sigma)$ induces an involution $\pi$ on $\iinter{1,n}$ by $\pi(i)=j$ when $\bullet_i\sim\bullet_j$, and $\pi(i)=i$ if $\bullet_i$ survives. We shall refer to $\pi$ as the \emph{pairing induced by the annihilations}. 

We will usually replace notation $\bullet_i\sim\bullet_j$ by a more precise series of notations: 
if $\bullet_i\sim\bullet_j$ with $i<j$, we write
$\bullet_i\collide\bullet_j$, or redundantly $\go_i\collide\come_j$, when $v_i=+1$ and $v_j=-1$,  $\bullet_i\tto\bullet_j$ when $v_i=+1$ and $v_j=0$, and $\bullet_i\from\bullet_j$ symmetrically. Note that in all cases this notation refers to annihilation, not merely collision, i.e.\ it excludes the case where $\bullet_i$ and $\bullet_j$ take part in a triple collision but one of them survives. 

Additionally, we write $x\from\bullet_i$ (for $i\in\N$ and $x\in\R$) to indicate that $\bullet_i$ crosses location~$x$ from the right (i.e.\ $v_i=-1$, $x<x_i$, and $\bullet_i$ is not annihilated when or before it reaches $x$), and $x\ffrom\bullet_i$ if $\bullet_i$ is \emph{first} to cross location~$x$ from the right. Symmetrically, we write $\bullet_i\tto x$ and $\bullet_i\fto x$. 

For any interval $I\subset(0,\infty)$, and any condition $C$ on particles, we denote by $(C)_I$ the same condition for the process restricted to the set $I$, i.e.\ where all particles outside $I$ are removed at time $0$ (however, the indices of remaining particles are unaffected by the restriction). For short, we write $\{C\}_I$ instead of $\{(C)_I\}$, denoting the event that the condition $(C)_I$ is realized. 

\subsection{Skyline}
Finally, we introduce a decomposition of the configuration that plays a key role. 

\begin{figure}
\includegraphics[scale=0.7, page=2]{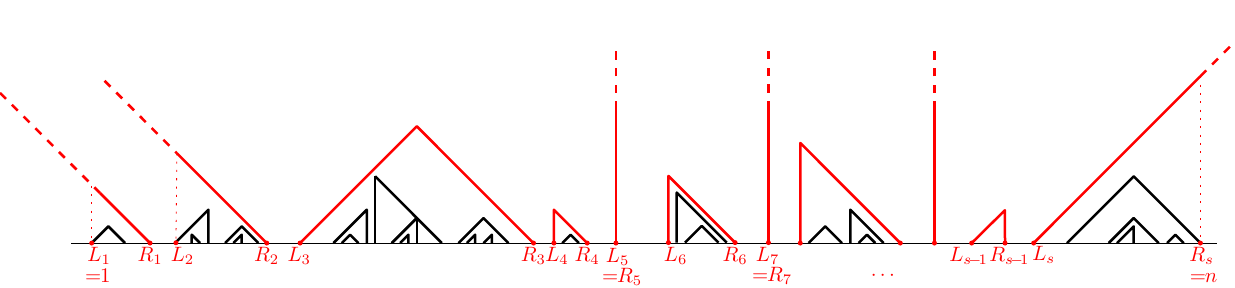}
\caption{Graphical space-time interpretation of the skyline of the process (see Definition~\ref{def:skyline}), here highlighted in red. }\label{fig:skyline}
\end{figure}

\begin{definition}\label{def:skyline}
The \emph{skyline} of the configuration $(v,s,\sigma)$ on $\iinter{1,n}$ is the family 
\[\sky_\ell(v,s,\sigma)=\big((L_1,R_1,\Sigma_1),\ldots,(L_S,R_S,\Sigma_S)\big)\] 
characterized by the following properties: $\iinter{L_1,R_1}$, $\iinter{L_2,R_2}$,\ldots, $\iinter{L_S,R_S}$ is a partition of $\iinter{1,n}$, for $k=1,\ldots,S$, the ``shape''  $\Sigma_k$ is one of the $6$ elements of the set $\{\uparrow,\nwarrow,\nearrow,\nearrow\uparrow,\uparrow\nwarrow,\nearrow\nwarrow\}$, and satisfies:
\begin{itemize}
	\item $\Sigma_k=\nwarrow$ if $v_{R_k}=-1$ and all particles indexed in $\iinter{L_k,R_k-1}$ are annihilated but $\bullet_{R_k}$ survives;
	\item $\Sigma_k=\uparrow$ if $v_{L_k}=0$, $R_k=L_k$, and $\bullet_{R_k}$ survives;
	\item $\Sigma_k=\nearrow\uparrow$ if $\go_{L_k}\tto\stay_{R_k}$, and no particle hits $x_{L_k}$ from the left, or $x_{R_k}$ from the right (hence $\bullet_{L_k}$ is the last to visit $[x_{L_k},x_{R_k})$);
	\item $\Sigma_k=\nearrow\nwarrow$ if $\go_{L_k}\collide\come_{R_k}$, and no particle hits $x_{L_k}$ from the left, or $x_{R_k}$ from the right;
\end{itemize}
and symmetrically for $\uparrow\nwarrow$ and $\nearrow$. 
\end{definition}

Since this definition amounts to splitting $\iinter{1,n}$ at indices of left- and right-going survivors and at endpoints of intervals never crossed by a particle (see Figure~\ref{fig:skyline}), the skyline is well-defined. 

This definition agrees with the notion of skyline introduced in~\cite{HST} to study the process on the full line (i.e.\ with particles indexed by $\Z$), in the supercritical regime. The extra shapes $\nwarrow$ and $\nearrow$ are however specific to the finite setting.

\subsection{Main results and organization of the paper}

Our main theorem is the following ``universality'' result. 

\begin{theorem}\label{thm:main}
For each $n$, the distribution of the skyline on $\iinter{1,n}$ does not depend on~$\ell$. 
\end{theorem}

A similar statement was given in~\cite{HST} in the infinite setting, with independent random lengths, where it only made sense in the supercritical regime. The above result highlights the combinatorial nature of this remarkable universality.

In order to underline the analogy with~\cite{BM19}, where universality of the number of surviving (i.e.\ non-annihilating) particles is proved in the context of generic velocities, let us phrase out an immediate yet sensibly weaker corollary. 

\begin{corollary}
For each $n$, the joint law of the number of surviving particles of respective velocity $-1$, $0$ and $+1$, does not depend on $\ell$. 
\end{corollary}

We give two proofs of the main result, of different nature and interest. 
\begin{itemize}
	\item First, in Section~\ref{sec:universality_induction}, we show by induction on the number of particles, that the probabilities $\PP_\ell(0\ffrom\come_n)$ and $\PP_\ell(\go_1\collide\come_n)$ do not depend on $\ell$, and that this implies the theorem. As a side result, this provides recursive formulae for these probabilities. These formulae were obtained in~\cite{HST} for random lengths, where they played a key role.
	\item Then, in Section~\ref{sec:universality_direct}, a direct proof is proposed, close in spirit to the proof of~\cite{BM19}, in that we study local invariance properties of the law of the skyline on the space $\Ell_n$ of lengths. Compared to~\cite{BM19}, the proof greatly simplifies thanks to the natural definition of the model at singular $\ell$, i.e.\ in a way of dealing with triple collisions that ensures continuity in law. 
\end{itemize}

This result has consequences for the classical case of random lengths. Let us already notice that it readily implies that the distribution of the skyline on $\iinter{1,n}$ does not depend on the distribution $m$ of interdistances. The same argument actually not only holds for i.i.d.\ but also for \emph{exchangeable} sequences $(\ell_1,\ldots,\ell_n)$. In the infinite setting, the extension of results such as the phase transition at $p_c=1/4$ (main theorem of~\cite{HST}) to exchangeable sequences could alternatively already be seen as a consequence of the universality in the i.i.d.\ case, and of de Finetti's theorem. 

The universality of the skyline also implies a previously unnoticed property of independence in the classical case of random lengths, which gives access to the Laplace transform of an interesting quantity: 

\begin{corollary}\label{cor:xA}
Consider ballistic annihilation with random lengths (on $\R_+$), and define the random variable
\[A=\min\{n\ge1\st 0\from\come_n\}.\]
\begin{enumerate}
[a)]
	\item The random variables $(A,x_A)$ and $(A,\widetilde x_A)$ have same distribution given $\{A<\infty\}$, where $\widetilde x$ is a copy of $x$ that is independent of $A$;
	\item Denote by $\Lr_\ell$ the Laplace transform of $x_1=\ell_1$, i.e.\ $\Lr_\ell(\lambda)=\EE[e^{-\lambda x_1}]$ for all $\lambda>0$. Then the Laplace transform $\Lr_D$ of $D\defeq x_A$ satisfies, on $\R_+$, 
\begin{equation}\label{eq:Laplace_D}
p\Lr_\ell \Lr_D^4-(1+2p)\Lr_\ell\Lr_D^2+2\Lr_D-(1-p)\Lr_\ell=0.
\end{equation}
\end{enumerate}
\end{corollary}

\bigskip

This corollary is proved after Theorem~\ref{thm:main} in the upcoming Section~\ref{sec:universality_induction}. Note in particular that $\Lr_\ell=\phi\circ\Lr_D$ where $\phi(w)=\frac{-2w}{pw^4-(1+2p)w^2-(1-p)}$, hence the distribution of $D$ characterizes the distribution $m$ of $\ell_1$, in complete contrast with the absence of dependence of $A$ with respect to $m$ (cf.~Proposition 4 of~\cite{HST}, which now follows from the above Theorem~\ref{thm:main}, cf.\ also Proposition~\ref{pro:universality_A} below). 

While we present a) as a corollary of Theorem~\ref{thm:main}, from which b) follows at once, we should also mention that a more direct proof of b) is possible (see the end of Section~\ref{sec:universality_induction}), from which a) could alternatively be deduced. This approach based on identification of Laplace transforms however gives no insight about the \latin{a priori} surprising Property a). 

The proof of Property a) will give the stronger statement of independence between the whole skyline on $\iinter{1,n}$ and $x_n$; this fact turns out to also hold in the asymmetric case and is therefore stated in Section~\ref{sec:asymmetric} as Theorem~\ref{thm:independence}, a). 

The interest in the random variable $D$ comes from the fact (mostly a consequence of ergodicity) that, for the process on the full line, the set of indices of static particles remaining at time $t$ has a density (in $\Z$) equal to 
\[c_0(t)=p\PP(D>t)^2.\] 
The value, and in particular the asymptotics for this quantity, which are important to understand the long-term behavior, can in principle be inferred from the Laplace transform of $D$. Such an analysis was conducted in~\cite{HST} without access to $\Lr_D$, under assumption of finite exponential moments for $\ell_1$, in order to enable approximating the tail of $D$ using that of $A$, which in turn could be addressed by combinatorial analytic methods on its generating series. Although the above computation is more explicit than that of~\cite{HST}, we refrain from stating more general asymptotics of $c_0(t)$ as $t\to\infty$, as these are not universal and would depend on a technical choice of further assumptions on the tail of $\ell_1$. Let us merely remark that, in the particular case when $\ell_1$ is exponentially distributed, i.e.\ $\Lr_\ell(\lambda)=(1+\lambda)^{-1}$, this confirms Equation (31) from~\cite{droz1995ballistic}: 
\[p\Lr_D(\lambda)^4-(2p+1)\Lr_D(\lambda)^2+2(\lambda+1)\Lr_D(\lambda)+p-1=0,\]
from which Laplace inversion in the asymptotic regime $\lambda\to0$ could be conducted, leading to asymptotics of $c_0(t)$, $t\to\infty$. 

We dedicate Section~\ref{sec:asymmetric} to a discussion of the contrasting lack of universality as soon as the distribution of velocities is not symmetric any more, thus raising a priori difficulties for explicit computations. Still, we give positive results (Theorem~\ref{thm:independence}), and in particular a remarkable property of gamma distributed interdistances that is insensitive to the asymmetry.

Finally, in Section~\ref{sec:monotonicity_A}, returning to the symmetric case and the universal distribution of the skyline, and in particular of $A$ (cf.~Corollary~\ref{cor:xA} above), which is in a sense explicit, we investigate its monotonicity properties with respect to the parameter $p$. While intuitively expected, and indeed observed numerically, these turn out to be complicated to establish in spite of the formulae at hand. We state a few conjectures and prove partial results. 

\section{Combinatorial universality -- Proof of the main results}\label{sec:universality_induction}

The main result will follow from the particular case in the proposition below, which is a stronger version of Theorem~2 from~\cite{HST}. 

\begin{proposition}\label{pro:universality_A}
For all $n$, and all $\ell\in\Ell_n$, define probabilities
\[p_n(\ell)=\PP_\ell(0\ffrom\come_n)\qquad\text{and}\qquad\delta_n(\ell)=\PP_\ell(\go_1\collide\come_n).\]
For all $n$, $p_n$ and $\delta_n$ do not depend on $\ell$, and are given by the following recursive equations, for $n\ge2$, 
\begin{gather}
p_n=\Bigl(p+\frac12\Bigr)\sum_{\substack{k_1+k_2\\=n-1}}p_{k_1}p_{k_2}-\frac p2\sum_{\substack{k_1+k_2+k_3+k_4\\=n-1}}p_{k_1}p_{k_2}p_{k_3}p_{k_4}\label{eqn:rec_pn}\\
\delta_n=\frac{1-p}2p_{n-1}-\frac p2\sum_{\substack{k_1+k_2+k_3\\=n-1}}p_{k_1}p_{k_2}p_{k_3},\label{eqn:rec_dn}
\end{gather}
with base cases $p_1=(1-p)/2$ and $\delta_1=0$. 
\end{proposition}

Note that $p_n=0$ for even $n$, and $\delta_n=0$ for odd $n$. 

Let us stress again that, as is the case in~\cite{BM19}, this ``universality'' with respect to $\ell$ does not follow from a direct coupling: indeed, the distribution of the full pairing induced by the annihilations is \emph{not} universal. A coupling between given vectors $\ell$ and $\ell'$ close enough is however possible, see Section~\ref{sec:universality_direct}, but wouldn't extend to a coupling in the \emph{infinite} random length setting, as any fluctuation in the lengths eventually breaks the coupling. Finally, as discussed in Section~\ref{sec:asymmetric}, the universality does not extend to the asymmetric case. 

\begin{proof}
The proof follows the main lines of Theorem~2 from~\cite{HST}, although using a stronger invariance. We proceed by induction on $n$, and show more generally that, for all $n$, none of the following probabilities depends on $\ell$ : 
\begin{gather*}
p_n(\ell)  = \PP_\ell(0\ffrom\come_n),\quad\alpha_n(\ell) = \PP_\ell(\{0\ffrom\come_n\}\cap\{\stay_1\}),\quad\beta_n(\ell)  =\PP_\ell(\{0\ffrom\come_n\}\cap\{\go_{1}\tto\stay\}),\\
\gamma_n(\ell)  =\PP_\ell(\{0\ffrom\come_n\}\cap\{\go_{1}\collide\come\}),\quad\text{and}\quad\delta_n(\ell) =\PP_\ell(\go_{1}\collide\come_n).
\end{gather*}
Since $p_n=\alpha_n+\beta_n+\gamma_n$, the proposition will follow.
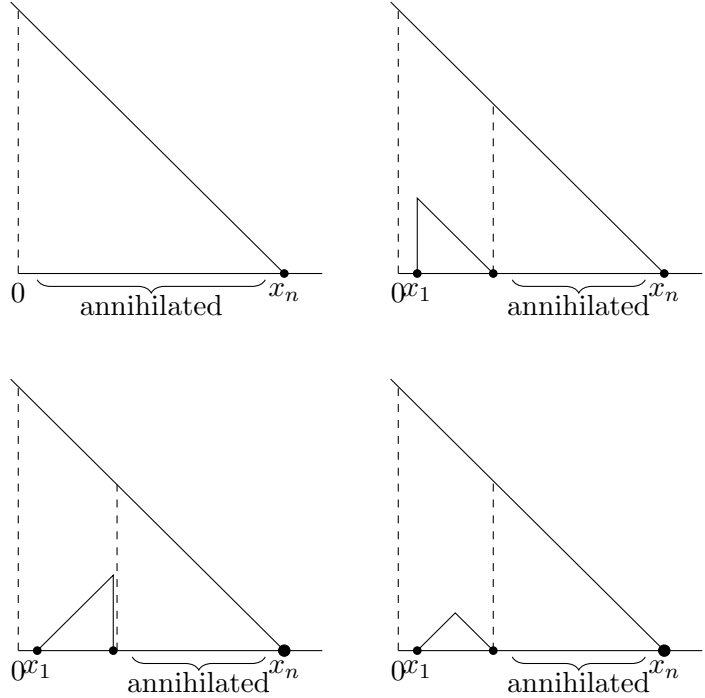
\begin{figure}
\begin{center}
\begin{tikzpicture}
\draw (0,5) node[anchor=north] {$0$} -- (4,5);
\draw[dashed] (0,5) -- (0,8.5);
\filldraw (3.5,5) circle [radius=0.05];
\node[anchor=north] at (3.5,5) {$x_n$};
\draw (3.5,5) -- (-0.1,8.6);
\draw [decorate,decoration={brace,amplitude=5pt,mirror,raise=0.5ex}]
  (0.25,5) -- (3.25,5) node[midway,yshift=-2.5ex]{annihilated};
  
\draw (5,5) node[anchor=north] {$0$} -- (9,5);
\draw[dashed] (5,5) -- (5,8.5);
\filldraw (8.5,5) circle [radius=0.05];
\node[anchor=north] at (8.5,5) {$x_n$};
\filldraw (5.25,5) circle [radius=0.05];
\node[anchor=north] at (5.25,5) {$x_1$};
\filldraw (6.25,5) circle [radius=0.05];
\draw (5.25,5) -- (5.25,6) -- (6.25,5);
\draw[dashed] (6.25,5) -- (6.25,7.25);
\draw (8.5,5) -- (4.9,8.6);
\draw [decorate,decoration={brace,amplitude=5pt,mirror,raise=0.5ex}]
  (6.5,5) -- (8.25,5) node[midway,yshift=-2.5ex]{annihilated};

\draw (0,0) node[anchor=north] {$0$} -- (4,0);
\draw[dashed] (0,0) -- (0,3.5);
\filldraw (3.5,0) circle [radius=0.075];
\node[anchor=north] at (3.5,0) {$x_n$};
\draw (3.5,0) -- (-0.1,3.6);
\filldraw (.25,0) circle [radius=0.05];
\node[anchor=north] at (.25,0) {$x_1$};
\filldraw (1.25,0) circle [radius=0.05];
\draw (.25,0) -- (1.25,1) -- (1.25,0);
\draw[dashed] (1.3,0) -- (1.3,2.2);
\draw [decorate,decoration={brace,amplitude=5pt,mirror,raise=0.5ex}]
  (1.5,0) -- (3.25,0) node[midway,yshift=-2.5ex]{annihilated};
  
\draw (5,0) node[anchor=north] {$0$} -- (9,0);
\draw[dashed] (5,0) -- (5,3.5);
\filldraw (8.5,0) circle [radius=0.075];
\node[anchor=north] at (8.5,0) {$x_n$};
\draw (8.5,0) -- (4.9,3.6);
\filldraw (5.25,0) circle [radius=0.05];
\node[anchor=north] at (5.25,0) {$x_1$};
\filldraw (6.25,0) circle [radius=0.05];
\draw (5.25,0) -- (5.75,.5) -- (6.25,0);
\draw[dashed] (6.25,0) -- (6.25,2.25);
\draw [decorate,decoration={brace,amplitude=5pt,mirror,raise=0.5ex}]
  (6.5,0) -- (8.25,0) node[midway,yshift=-2.5ex]{annihilated};
\end{tikzpicture}
\end{center}
\caption{The event corresponding to $p_n$ (top left) and decomposition into the events corresponding to $\alpha_n$ (top right), $\beta_n$ (bottom left) and $\gamma_n$ (bottom right).}
\end{figure}
\par We have $p_1=(1-p)/2$, and $\alpha_1=\beta_1=\gamma_1=\delta_1=0$. Let $n\in\N$ be such that $n\ge2$, and assume that the previous property holds up to the value $n-1$. 

First, let us consider $\alpha_n$ and prove more precisely that, for each integer $k$ in $\{2,\ldots,n-1\}$, $\PP_\ell(0\ffrom\come_n,\,\stay_1\from\come_k)$ is universal (i.e., does not depend on $\ell$). Let $k$ be such an integer. Note that we have equivalently
\begin{equation}
\{0\ffrom\come_n,\,\stay_1\from\come_k\}=\{\stay_1\}\cap\{x_1\ffrom\come_k\}_{(x_1,x_k]}\cap\{x_k\ffrom\come_n\}_{(x_k,x_n]}.\label{eq:case1}
\end{equation}
Let $\{r_1,\ldots,r_{k-1}\}$ and $\{s_1,\ldots,s_{n-k}\}$ be any disjoint subsets of $\{1,\ldots,n\}$ such that $r_1<\cdots<r_{k-1}$ and $s_1<\cdots<s_{n-k}$. Then, by standard properties of uniform permutations, conditional on the event
\[E_{r,s}=\big\{\sigma(\{2,\ldots,k\})=\{r_1,\ldots,r_{k-1}\}\text{ and }\sigma(\{k+1,\ldots,n\})=\{s_1,\ldots,s_{n-k}\}\big\},\] 
the random variables $(\ell_{\sigma(2)},\ldots,\ell_{\sigma(k)})$ and $(\ell_{\sigma(k+1)},\ldots,\ell_{\sigma(n)})$ are independent, and respectively distributed as $\bigl(\ell^{(r)}_{\tau(1)},\ldots,\ell^{(r)}_{\tau(k-1)}\bigr)$ and $\bigl(\ell^{(s)}_{\pi(1)},\ldots,\ell^{(s)}_{\pi(n-k)}\bigr)$, where $\ell^{(r)}_i=\ell_{r_i}$ and $\ell^{(s)}_i=\ell_{s_i}$, and $\tau,\pi$ are independent uniform permutations of $\{1,\ldots,k-1\}$ and $\{1,\ldots,n-k\}$ respectively. In particular, from~\eqref{eq:case1}, 
\begin{align*}
\PP_\ell(0\ffrom\come_n,\,\stay_1\from\come_k\s E_{r,s})
	& = p\,\PP_{\ell^{(r)}}(0\ffrom\come_{k-1})\PP_{\ell^{(s)}}(0\ffrom\come_{n-k})\\
	& = p\,p_{k-1}\bigl(\ell^{(r)}\bigr)p_{n-k}\bigl(\ell^{(s)}\bigr),
\end{align*}
which by induction does not depend on $\ell$. Since $E_{r,s}$ does not depend on $\ell$, summing over values of $r,s$, and of $k$ proves universality for $\alpha_n$. 

Secondly, we consider $\beta_n$. Contrary to the previous case, we only show that, for each integer $k$ in $\{2,\ldots,n-1\}$, the sum $\PP_\ell(0\ffrom\come_n,\,\go_1\tto\stay_k)+\PP_\ell(0\ffrom\come_n,\,\go_1\tto\stay_{k'})$ is universal, where $k'=n+1-k$. Summing over $k$ then gives universality of $2\beta_n$, hence of $\beta_n$. Let $k\in\{2,\ldots,n-1\}$, and note that
\begin{align*}
\{0\ffrom\come_n,\,\go_1\tto\stay_k\}
	& = \{\go_1\fto x_k\}_{[x_1,x_k)}\cap\{\stay_k\}\cap\{x_k\ffrom \come_n\}_{(x_k,x_n]}\\
	& \cap\Big(\{x_k-x_1<x_n-x_k\}\cup\big(\{x_k-x_1=x_n-x_k\}\cap\{s_k=-1\}\big)\Big).
\end{align*}
Conditional on the same type of event $E_{r,s}$ as above, the distances $x_k-x_1 = \sum_{1\le i\le k-1}\ell^{(r)}_i$ and $x_n-x_k = \sum_{1\le i\le n-k}\ell^{(s)}_i$ become deterministic, and as in the previous case the four events in the above conjunction are independent so that
\[\PP_\ell(0\ffrom\come_n,\,\go_1\tto\stay_k\s E_{r,s})
	= p_{k-1}\bigl(\ell^{(r)}\bigr)\,p\,p_{n-k}\bigl(\ell^{(s)}\bigr)\Bigl(\indic_{\bigl(\sum\ell^{(r)}<\sum\ell^{(s)}\bigr)}+\tfrac12\indic_{\bigl(\sum\ell^{(r)}=\sum\ell^{(s)}\bigr)}\Bigr).\]
Symmetrically, recalling that $k'=n+1-k$ and denoting by $E'_{s,r}$ the event that $\sigma(\{2,\ldots,k'\})=\{s_1,\ldots,s_{k'-1}\}$ and $\sigma(\{k'+1,\ldots,n\})=\{r_1,\ldots,r_{n-k'}\}$, we have
\[\PP_\ell(0\ffrom\come_n,\,\go_1\tto\stay_{k'}\s E'_{s,r})
	= p_{k'-1}\bigl(\ell^{(s)}\bigr)\,p\,p_{n-k'}\bigl(\ell^{(r)}\bigr)\Bigl(\indic_{\bigl(\sum\ell^{(s)}<\sum\ell^{(r)}\bigr)}+\tfrac12\indic_{\bigl(\sum\ell^{(s)}=\sum\ell^{(r)}\bigr)}\Bigr).\]
Thus, by summation, for any $r,s$ we have
\[\PP_\ell(0\ffrom\come_n,\,\go_1\tto\stay_k\s E_{r,s})+\PP_\ell(0\ffrom\come_n,\,\go_1\tto\stay_{k'}\s E'_{s,r})
	= p_{k-1}\bigl(\ell^{(r)}\bigr)\,p\,p_{n-k}\bigl(\ell^{(s)}\bigr).\]
By induction, this does not depend on $\ell$. Summing on values of $(r,s)$ yields the expected conclusion. 

Let us now consider $\gamma_n$. Here we have again that, for $1<k<n$, the probability $\PP_\ell(0\ffrom\come_n,\,\go_1\collide\come_k)$ is universal. Note indeed that
\begin{equation}
\{0\ffrom\come_n,\,\go_1\collide\come_k\}=\{\go_1\collide\come_k\}_{[x_1,x_k]}\cap\{x_k\ffrom\come_n\}_{(x_k,x_n]},
\end{equation}
and that, for all $(r,s)$ as before, 
\begin{align*}
\PP_\ell(0\ffrom\come_n,\,\go_1\collide\come_k\s E_{r,s})
	& =\PP_{\ell^{(r)}}(\go_1\collide\come_k)\PP_{\ell^{(s)}}(0\ffrom\come_n-k)\\
	& = \delta_{k-1}\bigl(\ell^{(r)}\bigr)p_{n-k}\bigl(\ell^{(s)}\bigr),
\end{align*}
which by induction does not depend on $\ell$. 

Finally, we are left with $\delta_n$. Note that $\{\go_1\collide\come_n\}$ implies 
$\{\go_1\}\cap\{x_1\ffrom\come_n\}_{(x_1,x_n]}$. Furthermore, the difference between these two events is precisely given by configurations where in absence of $\bullet_1$ we would have $0\ffrom\come_n$, but $\go_1$ actually collides with some static particle $\stay_i$, thereby freeing a particle $\come_j$ (compared to the configuration without $\bullet_1$) that becomes the first to hit $0$, followed by $\come_n$. This can be expressed as follows:
\[\{\go_1\}\cap\{x_1\ffrom\come_n\}_{(x_1,x_n]}
	= \{\go_1\collide\come_n\}\cup\bigcup_{1<i<j<n}\Bigl(\{0\ffrom\come_j\}\cap\{\go_1\tto\stay_i\}\cap\{x_j\ffrom\come_n\}_{(x_j,x_n]}\Bigr).\]
Since these events are disjoint, we get (furthermore summing over $i$)
\[\frac{1-p}2 \PP_\ell\bigl((x_1\ffrom\come_n)_{(x_1,x_n]}\bigr) = \delta_n + \sum_{1<j<n}\PP_\ell\bigl(0\ffrom\come_j,\, \go_1\tto\stay,\,(x_j\ffrom\come_n)_{(x_j,x_n]}\bigr). \]
Conditioning on $\sigma(1)$ and applying the induction assumption on $(\ell_{\sigma(i)})_{i=2,\ldots,n}$ shows that the probability on the left-hand side equals $p_{n-1}$ and is universal. As for the right-hand side probability, conditional on $\sigma(\{1,\ldots,j\})$, the events $\{0\ffrom\come_j,\,\go_1\tto\stay\}$ and $\{x_j\ffrom\come_n\}_{(x_j,x_n]}$ are independent and we get
\[\PP_{\ell}\bigl(0\ffrom\come_j,\, \go_1\tto\stay,\,(x_j\ffrom\come_n)_{(x_j,x_n]}\bigm| \sigma(\{1,\ldots,j\})\bigr)= \beta_{j-1}p_{n-j+1},\]
which by induction is universal. As a consequence, $\delta_n$ is universal too. Gathering identities proved along the way, we have
\begin{gather*}\alpha_n=p\sum_{1<k<n}p_{k-1}p_{n-k};\\
\beta_n=\frac12\alpha_n;\\
\gamma_n=\sum_{1<k<n}\delta_kp_{n-k};\\
\delta_n=\frac{1-p}{2}p_{n-1}-\sum_{1<k<n}\beta_kp_{n-k}.\end{gather*}
It follows that $p_n=\alpha_n+\beta_n+\gamma_n$ satisfies
\begin{align*}
p_n&=\frac32p\sum_{1<k<n}p_{k-1}p_{n-k}+\sum_{1<k<n}p_{n-k}\biggl(\frac{1-p}{2}p_{k-1}-\sum_{1<j<k}\beta_jp_{k-j}\biggr)\\
&=\Bigl(p+\frac12\Bigr)\sum_{1<k<n}p_{k-1}p_{n-k}-\frac p2\sum_{1<k<n}\sum_{1<j<k}\sum_{1<i<j}p_{n-k}p_{k-j}p_{j-i}p_{i-1},
\end{align*}
and the formula for $\delta_n$ comes analogously.
\end{proof}

From there, the distribution of the skyline follows: 

\begin{proof}[Proof of Theorem~\ref{thm:main}]
Let $s\in\N^*$, and let $(l_1,r_1,\varsigma_1),\ldots,(l_s,r_s,\varsigma_s)$ be any possible skyline on $\{1,\ldots,n\}$, in other words this sequence satisfies
\begin{itemize}
	\item $1=l_1\le r_1=l_2-1 < r_2=l_3-1<\cdots <r_{s-1}=l_{s-2}+1< r_s=n$;
	\item $r_i-l_i\begin{cases} 
=0 & \text{if $\varsigma_i=\uparrow$,}\\
\text{is odd} & \text{if $\varsigma_i\in\{\nearrow\uparrow,\uparrow\nwarrow\}$,}\\
\text{is even and $\ge2$ } & \text{if $\varsigma_i\in\{\nwarrow,\nearrow\}$;}
\end{cases}$
	\item $\varsigma_i=\nwarrow$ may only happen at the beginning, i.e.\ for all $i=1,\ldots,i_0$ (for some $i_0\ge0$);
	\item $\varsigma_i=\nearrow$ may only happen at the end, i.e.\ for all $i=n-j_0+1,\ldots,n$ (for some $j_0\ge0$).
\end{itemize}
Then we simply observe that the realization of the skyline reduces to events on the disjoint intervals $\iinter{l_i,r_i}$, which are independent:
\[
\PP\big(\sky_\ell(v,s,\sigma)=(l_i,r_i,\varsigma_i)_{1\le i\le s}\big)
	= \prod_{i=1}^s q^{(\ell)}_{\varsigma_i}(r_i-l_i),
\]
where, for all $m\in\N$, 
\begin{align*}
q^{(\ell)}_\uparrow(m) & = p \indic_{(m=0)},\\
q^{(\ell)}_{\nearrow\uparrow}(m) =q^{(\ell)}_{\uparrow\nwarrow}(m) & = p\PP(0\ffrom\come_{m-1})= p\, p_{m-1},\\
q^{(\ell)}_{\nearrow}(m)=q^{(\ell)}_{\nwarrow}(m) & = \PP(0\ffrom\come_m) = p_m,\\
q^{(\ell)}_{\nearrow\nwarrow}(m) & =\PP(\go_1\collide\come_m)=\delta_m,
\end{align*}
referring to notations $p_m$ and $\delta_m$ from the proof of Proposition~\ref{pro:universality_A}, where they are proved not to depend on $\ell$, thereby implying the theorem. 
\end{proof}

Corollary~\ref{cor:xA} is finally a direct corollary of Proposition~\ref{pro:universality_A}: 

\begin{proof}[Proof of Corollary~\ref{cor:xA}]
Property a) is equivalent to saying that, for all $n$, the random variables $\indic_{(A=n)}$ and $x_n$ are independent. Indeed, for all $n\in\N$ and $t>0$, $\PP(x_A\le t,\,A=n)=\PP(x_n\le t,\,A=n)$ while $\PP(\widetilde x_A\le t,\,A=n)=\PP(x_n\le t)\PP(A=n)$. Let $n\in\N$. Conditional on $\ell^{(\cdot)}=\bigl(\ell^{(1)},\ldots,\ell^{(n)}\bigr)$, which are the order statistics of $(\ell_1,\ldots,\ell_n)$, we have $x_n=\ell^{(1)}+\cdots+\ell^{(n)}$ and we reduce to the finite setting of Theorem~\ref{thm:main}:
\[\PP(x_n\le t,\,A=n)=\EE\bigl[\PP_{\ell^{(\cdot)}}(0\ffrom\come_n)\indic_{\bigl(\ell^{(1)}+\cdots+\ell^{(n)}\le t\bigr)}\bigr].\]
Since, by Theorem~\ref{thm:main}, the probability on the right-hand side does not depend on $\ell^{(\cdot)}$, and thus equals $\PP(A=n)$, this concludes a). 

Property b) follows: for all $\lambda>0$, since $x_n=\ell_1+\cdots+\ell_n$,
\[\Lr_D(\lambda)=\EE[e^{-\lambda x_A}]=\EE[e^{-\lambda \widetilde x_A}]=\sum_{n=0}^\infty\EE[e^{-\lambda x_n}]\PP(A=n)=\sum_{n=0}^\infty\EE[e^{-\lambda\ell_1}]^n\PP(A=n)=f(\Lr_\ell(\lambda)),\]
where $f$ is the generating function of $A$. 
The following relationship satisfied by $f$ was deduced in the restricted setting of i.i.d.\ interdistances in \cite{HST}, but we repeat it here for completeness. Introducing the generating series
\[
A:x\mapsto \sum_{n=0}^\infty \alpha_n x^n, \quad
B:x\mapsto  \sum_{n=0}^\infty \beta_n x^n, \quad
C:x\mapsto  \sum_{n=0}^\infty \gamma_n x^n, \quad
\text{and }
D:x\mapsto  \sum_{n=0}^\infty \delta_n x^n,
\]
for which the recurrence relations proved above yield the relationships
\begin{gather*}
A(x)=px f(x)^2,\qquad B(x)=\frac12A(x),\\
C(x)=D(x)f(x),\quad\text{and}\quad D(x)=\frac{1-p}2xf(x)-B(x)f(x),
\end{gather*}
and using the fact that $f(x)=\frac{1-p}2x+A(x)+B(x)+C(x)$, we obtain
\[f(x)=\frac{1-p}2x+\frac32pxf(x)^2+\frac{1-p}2xf(x)^2-\frac12 pxf(x)^4,\]
from which b) follows.
\end{proof}

Deducing a) from b) would actually require computing the joint transform $\Lr_{(A,x_A)}(s,\lambda)=\EE[s^A e^{-\lambda x_A}]$: one gets similarly $\Lr_{(A,x_A)}(s,\lambda)=f(s\Lr_\ell(\lambda))$, which coincides with $\Lr_{(A,\widetilde x_A)}(s,\lambda)$.

As advertised after the statement of Corollary~\ref{cor:xA}, a more direct computation of the Laplace transform is possible, that we sketch below. 

\begin{proof}[Alternative proof of Corollary~\ref{cor:xA} b)]\label{proof_laplace}
Let us directly obtain the functional equation~\eqref{eq:Laplace_D} for the Laplace transform $\Lr_D(\lambda)$, using a continuous counterpart to the induction obtained for $A$. For a change, we shall also partly refer to the mass transport principle, whose use in the context of ballistic annihilation was introduced by Junge and Lyu~\cite{junge-lyu-asymmetric}. 
Let us warn the reader that, since the arguments rely on similar arguments given elsewhere in the paper, and this is an alternative proof, we give somewhat fewer details. 
Let $\lambda\ge0$.

We split the expectation according to the type of collision $\bullet_1$ is involved in. 
First, if $v_1=-1$ then $D=x_1=\ell_1$: 
\[\EE[e^{-\lambda D}\indic_{(\come_1)}]=\EE[e^{-\lambda\ell_1}\indic_{(\come_1)}]=\Lr_\ell(\lambda)\frac{1-p}2.\] 
Consider now the case when $v_1=0$.  
Define, in general, $D',D''$ by 
$x_1+D'=\min\{x_k\st (x_1\ffrom\come_k)_{(x_1,x_k]}\}$ and $x_1+D'+D''=\min\{x_k\st (x_1+D'\ffrom\come_k)_{(x_1+D',x_k]}\}$, 
so that $D'$ and $D''$ are independent copies of $D$, and are independent of $v_1$. Conditional on $v_1=0$, we have $D=x_1+D'+D''$: either $D=\infty$, in which case either $D'=\infty$ or $D''=\infty$, or $D<\infty$, which implies that $\stay_1$ is first annihilated, and then $0$ is hit by a particle that hit $x_1+D'$ before. Thus, 
\[\EE[e^{-\lambda D}\indic_{(\stay_1)}]
	= \EE[e^{-\lambda(x_1+D+D')}\indic_{(\stay_1)}]=\EE[e^{-\lambda\ell_1}]\EE[e^{-\lambda D}]^2 p = \Lr_\ell(\lambda)\Lr_D(\lambda)^2 p.
\]
We focus now on $\EE[e^{-\lambda  D}\indic_{(\go_1\tto\stay)}] = \EE[e^{-\lambda D}\indic_{(\go_1\tto\stay)\wedge(D<\infty)}].$ Let us apply the mass transport principle (cf.~\cite{junge-lyu-asymmetric}) to $f(u,v)=e^{-\lambda D(u,v)}\indic_{(\go_u\tto\stay_v)}$, where $D(u,v)$ is the sum of the distances from $x_v$ to the first particle to cross $x_v$ from the left (which is $x_v-x_u$ on the event $\{\go_u\tto\stay_v\}$), and to the first particle to cross $x_v$ from the right. Thus, we temporarily extend to process to the full-line $\R$ in order to use the mass transport principle. This gives:
\[\sum_{v\in\Z}\EE[e^{-\lambda D(0,v)}\indic_{(\go_0\tto\stay_v)}] = \sum_{u\in\Z}\EE[e^{-\lambda D(u,0)}\indic_{(\go_u\tto\stay_0)}],\]
which rewrites as follows (using translation invariance):
\[\EE[e^{-\lambda (D-\ell_1)}\indic_{(\go_1\tto\stay)\wedge(D<\infty)}]
	= p\EE[e^{-\lambda (D_-+D)}\indic_{(D_-<D)}]\]
where $D_-$ is the symmetric counterpart to $D$ on $\R_-$; it is an independent copy of $D$, hence, since also $\ell_1$ is independent of $D-\ell_1$ (actually $\ell_1$ plays no role in the collisions),
\[\EE[e^{-\lambda D}\indic_{(\go_1\tto\stay)}]=\EE[e^{-\lambda\ell_1}]\frac12\EE[e^{-\lambda D}]^2=\frac p2\Lr_\ell(\lambda)\Lr_D(\lambda)^2. \]

Let us finally consider $\EE[e^{-\lambda D}\indic_{(\go_1\collide\come)}]$. Conditionally on the event $\{\go_1\collide\come\}$, $D=\ell_1+\Delta+D'$ where $\Delta=x_K-x_1$ if $K$ is the index such that $\go_1\collide\come_K$, and $D'=D-x_K$. Note that $\ell_1$, $\Delta$ and $D'$ are independent on that event, and that $D'$ has same distribution as $D$ so that
\begin{align*}
\EE[e^{-\lambda D}\indic_{(\go_1\collide\come)}]
	& =\EE[e^{-\lambda\ell_1}]\EE[e^{-\lambda D}]\EE[e^{-\lambda\Delta}\indic_{( \go_1\collide\come)}]\\
	& =\Lr_\ell(\lambda)\Lr_D(\lambda)\EE[e^{-\lambda\Delta}\indic_{(\go_1\collide\come)}].
\end{align*}

We are thus left with the Laplace transform of $\Delta$. As in the study of the law of $A$, we notice that $\{\go_1\collide\come\}$ happens on $\{\go_1\}\cap\{x_1\ffrom\come\}_{(x_1,+\infty)}$ \emph{unless} $\go_1$ hits a static particle first, i.e.\ there exists $j<k$ such that $\{\go_1\tto\stay_j\}\cap\{0\ffrom\come_k\}_{(0,+\infty)}\cap\{x_k\ffrom\come\}_{(x_k,+\infty)}$ is realized. The last condition of that event is independent of the previous ones, and depends on a piece of environment of length distributed as $D$. Thus we find
\begin{align*}
\EE[e^{-\lambda\Delta}\indic_{(\go_1\collide\come)}]
	& = \frac{1-p}2\Lr_D(\lambda)-
	\EE[e^{-\lambda (D-\ell_1)}\indic_{(\go_1\tto\stay)\wedge(0\ffrom\come)}]\Lr_D(\lambda)\\
	& = \frac{1-p}2\Lr_D(\lambda)-\frac p2\Lr_D(\lambda)^3
\end{align*}
where for the last computation we reuse the previous case $\EE[e^{-\lambda D}\indic_{(\go_1\tto\stay)}]$.

All together, this gives
\begin{align*}
\Lr_D(\lambda)=\frac{1-p}2\Lr_\ell(\lambda) & +p\Lr_\ell(\lambda)\Lr_D(\lambda)^2+\frac p2\Lr_\ell(\lambda)\Lr_D(\lambda)^2\\
	& +\Lr_\ell(\lambda)\Lr_D(\lambda)\Bigl(\frac{1-p}2\Lr_D(\lambda)-\frac p2\Lr_D(\lambda)^3\Bigr),
\end{align*}
which after simplification is the claimed identity. 
\end{proof}

Let us merely mention that the proof could also give the joint transform of $A$ and $D=x_A$, from which part a) of Corollary~\ref{cor:xA} follows as well.

\section{Direct approach to universality}\label{sec:universality_direct}

We aim here at giving a direct proof of the fact that the law of the skyline does not depend on the sequence of interdistances. This approach unites the model for different values of $\ell$ and gives a clearer understanding of the universality property but does not however yield explicit distributions. Let $n$ be fixed in this part. 

Among the set $\Ell_n$ of lengths, we distinguish ``generic'' length sequences, for which no triple collision may happen, in other words no two subsets of lengths have the same sum:
\[\Ellgen = \{\ell=(\ell_1,\ldots,\ell_n)\in\Ell_n\st \forall I,J\subset\{1,\ldots,n\}, \ell_I\ne\ell_J\text{ unless }I=J\},\]
where for any subset $I\subset\{1,\ldots,n\}$, we let $\ell_I=\sum_{i\in I}\ell_i$. Finally, we will need to refer to lengths allowing a \emph{single} triple collision: 
\[\Ellsin = \{\ell\in\Ell_n\st \text{there is a unique pair } I,J\subset\{1,\ldots,n\} \text{ such that }I\cap J=\emptyset\text{ and } \ell_I=\ell_J\},\]
and $\Ellmul=\Ell_n\setminus(\Ellgen\cup\Ellsin)$. 

\textbf{Locally constant on $\Ellgen$.} Notice first that the joint law of the velocities and pairing among annihilating particles is locally constant on $\Ellgen$. Indeed all relative orders among values of $\ell_I$, for $I\subset\{1,\ldots,n\}$, are locally constant on $\Ellgen$, hence for any velocities $v_1,\ldots,v_n\in\{-1,0,+1\}$ and any involution $\pi:\{1,\ldots,n\}\to\{1,\ldots,n\}$, the set of permutations $\sigma$ producing the pairing $\pi$ (i.e.\ such that $\bullet_i\sim\bullet_j$ if and only if $\pi(i)=j$, and $\pi(i)=i$ if and only if $\bullet_i$ survives) is itself locally constant in $\Ellgen$. As an immediate consequence, the law of the skyline is constant on each connected component of $\Ellgen$. 

\textbf{Continuous on $\Ellsin$.} Let us argue that these probabilities are continuous on $\Ellsin$. Let $\ell\in\Ellsin$. There is thus a unique pair $I,J$ such that $I\cap J=\emptyset$ and $\ell_I=\ell_J$. 

If $\eps$ denotes the smallest difference among $\lvert{\ell_K-\ell_L}\rvert$ for all $K,L$, $K\cap L=\emptyset$, $\{K,L\}\ne\{I,J\}$, then each vector $\ell'$ at uniform distance smaller than $\eps$ from $\ell$ either belongs to one of two connected components $\Ell_+(\ell)$ or $\Ell_-(\ell)$ of $\Ellgen$ according to whether $\ell'_I<\ell'_J$ or $\ell'_I>\ell'_J$, or to $\Ell_0(\ell)\subset\Ellsin$ if $\ell'_I=\ell'_J$. 

If $\ell'\in\Ell_0(\ell)$, the joint law of velocities, spins and pairing is preserved as above, hence the probability is the same as for $\ell$. 

If $\ell'\in\Ell_+(\ell)$, let us establish that the probability is preserved. Let us describe a one-to-one map $\Phi$ on the set of velocities, spins and permutations $(v,s,\sigma)$ that preserves the number of static (hence of moving) particles and such that $\sky_\ell(v,s,\sigma)=\sky_{\ell'}(\Phi(v,s,\sigma))$ (however, contrary to the previous cases, $\Phi$ does not a priori preserve the pairing); since the probability of each realization of $(v,s,\sigma)$ only depends on the number of static particles (thanks to symmetry), this will conclude the argument. The skyline for $\ell$ does not depend on the pairing of a given subset of $\{1,\ldots,n\}$ given that it totally annihilates and lies below a given annihilating pair or a surviving $\pm1$ particle. It is therefore sufficient that $\Phi$ only alters such subsets. 

If $(v,s,\sigma)$ has no triple collision for $\ell$, then $\Phi(v,s,\sigma)=(v,s,\sigma)$, and the pairings for $\ell$ and $\ell'$ are the same, as for $\Ellgen$, so that $\sky_\ell(v,s,\sigma)=\sky_{\ell'}(\Phi(v,s,\sigma))$. 

If there is a triple collision $\go_j\tto\stay_i\from\come_k$ for $(v,s,\sigma)$ and $\ell$, then we let $(v',s',\sigma')=\rev^{\ell,\ell'}_{j,k}(v,s,\sigma)$. Here, in wider generality,
for $1\le j<k\le n$, $\rev^{\ell,\ell'}_{j,k} :(v,s,\sigma)\mapsto (v',s',\sigma')$ is a ``reversing operator around a triple collision between $\go_j$ and $\come_k$'',
from $\{-1,0,+1\}^n\times\{-1,+1\}^n\times\perm_n$ to itself, defined by: (see also Figure~\ref{fig:rev})
\begin{itemize}
	\item if $(v,s,\sigma)$ does not induce $\go_j\tto\stay\from\come_k$ for $\ell$, then $(v',s',\sigma')=(v,s,\sigma)$. 
	\item otherwise,  i.e.\ if $\go_j\tto\stay_i\from\come_k$ for some $j<i<k$, for distances $\ell$, then denote $I=\sigma(\{j,\ldots,i-1\})$ and $J=\sigma(\{i,\ldots,k-1\})$ (so that $\ell_I=\ell_J$), and 
	\begin{itemize}
		\item if $\ell'_I<\ell'_J$ and $s_i=-1$, or $\ell'_I>\ell'_J$ and $s_i=+1$, then $(v',s',\sigma')=(v,s,\sigma)$;
		\item else, define $(v',s',\sigma')$ from $(v,s,\sigma)$ by mirroring the interval $(x_j,x_k)$, i.e.\ reversing the order of interdistances, velocities and spins, and furthermore changing velocities and spins to their opposite, in this interval : 
for $j<m\le k$,  $\sigma'(m) =\sigma(j+1+k-m)$ and for $j<m<k$, $v'_m=-v_{j+k-m}$, $s'_m=-s_{j+k-m}$; and $(v',s',\sigma')$ and $(v,s,\sigma)$ coincide elsewhere. 
	\end{itemize}
\end{itemize}
Note that, in the last case, $(v',s',\sigma')$ still has a triple collision $\go_j\tto\stay_{i'}\from\come_k$ for $\ell$, at $i'=j+k-i=i+|J|-|I|$; the mirroring has the effect of exchanging the roles of $I$ and $J$, but also of letting $s_{i'}=-s_i$, so that $(v',s',\sigma')$ would also fall into this last case hence $\rev^{\ell,\ell'}_{j,k}(v',s',\sigma')=(v,s,\sigma)$. Thus this operator is involutive on the subset of all $(v,s,\sigma)$ that have a triple collision $\go_j\tto\stay\from\come_k$. For $\ell\in\Ellsin$, the configuration space is partitioned into those subsets for $1\le j<k\le n$; as a consequence, $\Phi$ is involutive, hence bijective, on the whole configuration space $\{-1,0,+1\}^n\times\{-1,+1\}^n\times\perm_n$.

And by construction, each operator $\rev_{j,k}^{\ell,\ell'}$, and thus $\Phi$, only affects the pairing of indices of particles that annihilate and are lying ``under'' a moving particle (i.e.\ whose range is later visited by another particle), which has no consequence on the skyline. In particular, $\sky_\ell(v,s,\sigma)=\sky_{\ell'}(v',s',\sigma')$. 

Finally, $\rev_{j,k}^{\ell,\ell'}$, hence $\Phi$, clearly preserves the number of static particles.

\begin{figure}[h!t]
\includegraphics[height=10cm,page=2]{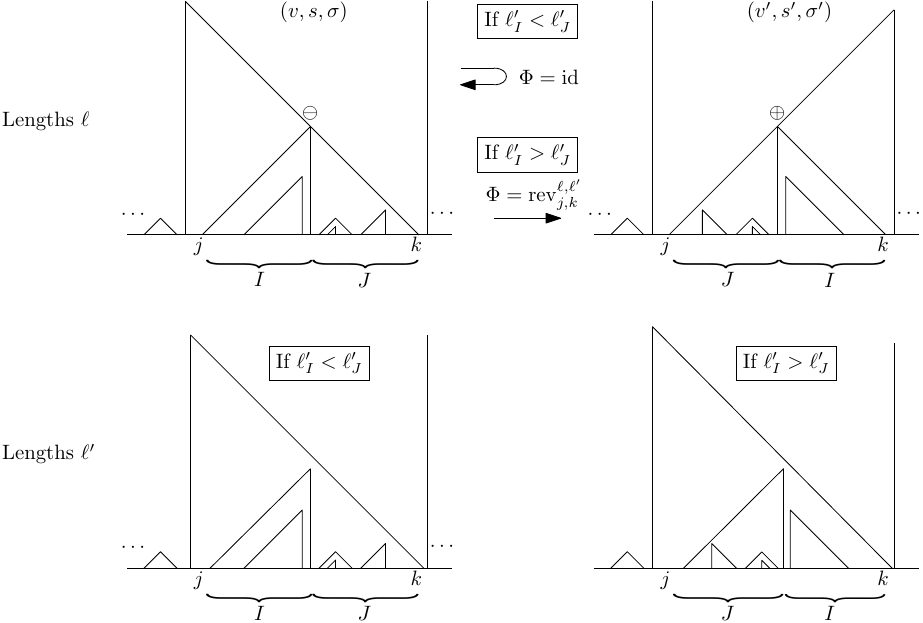}
\caption{Illustration of the map $\Phi$ in the case of a single triple collision.}\label{fig:rev}
\end{figure} 

\textbf{Connectedness of $\Ellgen\cup\Ellsin$.} We have that $\Ellgen\cup\Ellsin$ is a connected subset of $\Ell$. Indeed, its complement $\Ellmul$ in the positive full-dimensional cone $\Ell$ is a finite union of subspaces of codimension at least $2$, namely induced by at least two different constraints $\ell_{I_1}=\ell_{J_1}$, $\ell_{I_2}=\ell_{J_2}$ (which are not equivalent, hence non-colinear, due to $I_i\cap J_i=\emptyset$, $i=1,2$) . 

From the previous points, we conclude that the probabilities are constant on $\Ellgen\cup\Ellsin$. 

\textbf{Extension to $\Ellmul$.}
Let us finally extend the argument for $\Ellsin$ to the general case. We may pick $\ell'\in\Ellgen$ close enough to $\ell$ so that all relative orders among the sums $\ell_I$ are preserved except for the equality cases. We then describe a permutation $\Phi$ of triples $(v,s,\sigma)$ as above, which only depends on the order within all pairs $\ell'_I$ and $\ell'_J$ where $\ell_I=\ell_J$. Let $(v,s,\sigma)$ be a configuration. For that configuration, and distance $\ell$, we may order the triple collision pairs $(j,k)$ (i.e.\ such that $\go_j\tto\stay\from\come_k$) as $(j_1,k_1),\ldots,(j_M,k_M)$ in a way that complies with inclusion: if $1\le K\le L\le M$, then either $[j_K,k_K]\cap[j_L,k_L]=\emptyset$ or $[j_K,k_K]\subset[j_L,k_L]$. It suffices so first list, in arbitrary order, all intervals that are minimal for inclusion, and then iterate on the remaining ones. Then we define 
\[\Phi(v,s,\sigma)=\rev^{\ell,\ell'}_{j_M,k_M}\circ\rev^{\ell,\ell'}_{j_{M-1},k_{M-1}}\circ\cdots\circ\rev^{\ell,\ell'}_{j_1,k_1}(v,s,\sigma).\]
For $K=1,\ldots,M-1$, due to the ordering, the application of $\rev_{j_K,k_K}^{\ell,\ell'}$ in $\Phi$ does not alter the indices of the forthcoming triple collisions $(j_{K+1},k_{K+1}),\ldots,(j_M,k_M)$, ensuring that each operator really acts on a triple collision. Also, the ordering among triple collisions with disjoint supports has no effect on $\Phi$ since the corresponding $\rev$ operators commute. If there is no triple collision for $(v,s,\sigma)$ and $\ell$, we mean to define $\Phi(v,s,\sigma)=(v,s,\sigma)$. 

Although the operator $\Phi$ does not preserve the set of pairs $(j_1,k_1),\ldots,(j_M,k_M)$ of triple collisions, it does preserve the \emph{maximal} ones (for inclusion), which ensures it preserves the skyline when going from $(v,s,\sigma),\ell$ to $\Phi(v,s,\sigma),\ell'$. Also, it is still an involution: this is obtained by induction on the maximum number of nested triple collisions, together with the simple fact that $\Phi$ (and even each $\rev_{j,k}^{\ell,\ell'}$) commutes with a mirroring $\Sigma$ of the whole interval where $\Phi$ is acting. Also, $\Phi$ still preserves the number of static particles. Altogether, we deduce that the law of $\sky_\ell$ and $\sky_{\ell'}$ are equal. 

\section{Asymmetric case: failure of universality and a remarkable property of gamma distributions}\label{sec:asymmetric}

Let us consider the asymmetric case, where the distribution of velocities is given by $(1-r)(1-p)\delta_{-1}+p\delta_0+r(1-p)\delta_{+1}$, for some $r\in(0,1)\setminus\{\frac12\}$, and still $p\in(0,1)$. Many questions remain open in this case, but Junge and Lyu~\cite{junge-lyu-asymmetric} could still prove that some of the identities of the symmetric case can be extended, implying that the model still has a subcritical and a supercritical phase, although the phase transition was not proved unique. 

A notable difference, that helps understand why the asymmetric case might be sensibly harder, is the apparent lack of universality. It is indeed simple although tedious to check on the first cases that the distribution of $A$ is distribution-dependent. In particular, one can check that $\prb{A=5}$ depends on the distribution. Writing $z:=\prb{x_4-x_1>x_5-x_4}+\frac12\prb{x_4-x_1=x_5-x_4}$, we have
\begin{align*}
\prb{(A=5)\wedge(v=(1,1,0,0,-1))}&=zr^2(1-r)p^2(1-p)^3/8\\
\prb{(A=5)\wedge(v=(1,0,0,-1,-1))}&=(1-z)r(1-r)^2p^2(1-p)^3/8\\
\prb{(A=5)\wedge(v=(1,0,-1,0,-1))}&=zr(1-r)^2p^2(1-p)^3r^2(1-r)/16\\
\prb{(A=5)\wedge(v=(1,0,1,0,-1))}&=(1-z)r^2(1-r)p^2(1-p)^3/16,
\end{align*}
but the total contribution from other possible velocities is universal. Since $r\neq\frac12$, the total of these four probabilities depends on $z$ and hence on the distribution of distances.

Let us still state two surprising properties that hold in the asymmetric case. The first one would, in the symmetric case, follow at once from Theorem~\ref{thm:main} by a seamless generalization of the proof of Corollary~\ref{cor:xA} a). The second one however is new in any case.

\begin{theorem}\label{thm:independence}
We consider the random lengths setting. 
\begin{enumerate}[a)]
	\item In the symmetric or asymmetric cases, for all $n$, $x_n$ and $\sky_{(\ell_1,\ldots,\ell_n)}$ are independent.
	\item In the symmetric or asymmetric cases and if $m$ is a gamma distribution, for all $n$, $x_n$ and the whole combinatorial configuration, i.e.~$(v,\pi)$ (velocities and pairing), are independent. 
\end{enumerate}
\end{theorem}

Note that, due to the assumed independence between $\ell$ and $v$, the property b) is actually an independence between $x_n$ and $\pi$ given any velocities $v_1,\ldots,v_n$.
 
Let us give a simple counterexample illustrating why b) doesn't hold in general. Consider $m=\frac12\delta_1+\frac12\delta_4$ and the configuration on five particles given by velocities $v=(1,1,0,0,-1)$ and pairing $\pi=\begin{pmatrix}4& 3& 2& 1& 5\end{pmatrix}$ (i.e., $\bullet_1\sim\bullet_4$, etc.). Then, given this realization of $(v,\pi)$, interdistances \emph{necessarily} are $\ell_2=\ell_3=\ell_4=1$ and $\ell_5=4$ (remember $\ell_1$ plays no role), see Figure~\ref{fig:counterexample}, hence the distribution of $x_5$ takes two values depending on $\ell_1$, and thus clearly differs from the unconditioned distribution. 

\begin{figure}[h!t]
\begin{center}
\includegraphics[height=2cm]{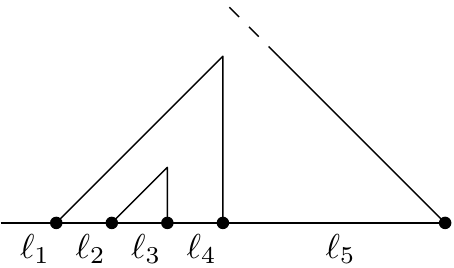}
\end{center}
\caption{Simple counterexample to the independence of $x_n$ from the pairing of $\iinter{1,n}$ for arbitrary interdistance distribution (see above)}\label{fig:counterexample}
\end{figure}

Let us remind the reader that the model of ballistic annihilation was studied in the physics literature under the assumption of exponential interdistances (see e.g.~\cite{droz1995ballistic}), which simplified computations. We don't have knowledge however of a previous result that would rely on that distribution except for technical reasons.

Let it finally be mentioned that we can't rule out a different form of universality, which might still await discovery. Still, numerical simulations suggest that the critical probability itself, should it exist, could depend on the distribution of interdistances. 

\begin{proof}
a) 
Note that conditioning on the skyline amounts to a conjunction of independent conditions on disjoint subintervals, that are either of the type $\{x_j\ffrom\come_k\}$ (including $\stay_j\from\come_k$) or $\{\go_j\collide\come_k\}$, and $x_n$ is the total length of these subintervals, together with unconditioned intervals in-between. 
It is therefore sufficient to show independence between $x_n$ and both of $\{0\ffrom\come_n\}$  and $\{\go_1\collide\come_n\}$. This property will be obtained via a similar recursion scheme as in the proof of Proposition~\ref{pro:universality_A} -- or rather as in the proof of Theorem 2 from~\cite{HST}, since we are considering random lengths. We actually prove the stronger statement of independence between $x_n$ and each of the events $\{0\ffrom\come_n\}\cap\{\stay_1\}$, $\{0\ffrom\come_n\}\cap\{\go_1\tto\stay\}$, $\{0\ffrom\come_n\}\cap\{\go_1\collide\come\}$ and $\{\go_1\collide\come_n\}$. 

The case $n=1$ is clear. Assume now $n\ge2$ and that the independences hold for any number $m< n$ of particles (note that, for each $m$, depending on parity, only one of the conditions $\{0\ffrom\come_m\}$ and $\{\go_1\collide\come_n\}$ has nonzero probability, so independence is trivial for the other). Symmetrically, we already remark that this assumption implies an independence between $x_m$ (unchanged by left-right symmetry) and the event $\{\go_1\fto x_{m+1}\}_{[x_1,x_{m+1})}$ for all $m<n$. 

In the following, in order to emphasize that we restrict to $[0,x_n]$, we denote $\PP^{(n)}$ the probability of the model restricted to particles $\bullet_1,\ldots,\bullet_n$ (remember the random length model was defined for infinitely many particles). 

Consider any measurable function $f:\R\to\R_+$. 
We have 
\[\EE^{(n)}[f(x_n)\indic_{(0\ffrom\come_n)}\indic_{(\stay_1)}]=\sum_{1<k<n}\EE^{(n)}[f(x_n)\indic_{(\stay_1\from\come_k)}\indic_{(x_k\ffrom\come_n)}]\]
and, under the condition appearing on the right hand side, by induction, each of $x_k-x_1$ and $x_n-x_k$ (and trivially $x_1$) have unconditioned distributions; they are also mutually independent, as in their joint unconditioned distribution, so that the distribution of $x_n=(x_n-x_k)+(x_k-x_1)+x_1$ is unaffected by this condition, hence \[\EE^{(n)}[f(x_n)\indic_{(0\ffrom\come_n)}\indic_{(\stay_1)}]=\EE^{(n)}[f(x_n)]\PP((0\ffrom\come_n)\wedge(\stay_1)),\]
as expected. Next, we have (as in the study of $\beta_n$ in the proof of Proposition~\ref{pro:universality_A}), denoting $k'=n+1-k$ for any $1<k<n$, 
\begin{align*}
\EE^{(n)}[f(x_n)\indic_{(0\ffrom\come_n)}\indic_{(\go_1\tto\stay)}]=\frac p2&\sum_{1<k<n}\Bigl(\EE^{(n)}[f(x_n)\indic_{(\go_1\fto x_k)}\indic_{(x_k\ffrom\come_n)}\indic_{(x_k-x_1<x_n-x_k)}]\\
	& +\EE^{(n)}[f(x_n)\indic_{(\go_1\fto x_{k'})}\indic_{(x_{k'}\ffrom\come_n)}\indic_{(x_{k'}-x_1<x_n-x_{k'})}]\\
	& +\frac12\EE^{(n)}[f(x_n)\indic_{(\go_1\fto x_{k})}\indic_{(x_{k}\ffrom\come_n)}\indic_{(x_{k}-x_1=x_n-x_{k})}]\\
	& +\frac12\EE^{(n)}[f(x_n)\indic_{(\go_1\fto x_{k'})}\indic_{(x_{k'}\ffrom\come_n)}\indic_{(x_{k'}-x_1=x_n-x_{k'})}]
\Bigr).
\end{align*}
By the induction assumption, for all $k$, conditional on the event $\{\go_1\fto x_{k'}\}$, $x_{k'}-x_1$ is unconditioned, and in particular (by the induction again, symmetrically) has same distribution as $x_n-x_{n-k'+1}=x_n-x_k$ conditional on $\{x_{k}\ffrom\come_n\}$; similarly, conditional on $\{x_{k'}\ffrom\come_n\}$, $x_n-x_{k'}$ has same distribution as $x_{n-k'+1}-x_1=x_k-x_1$ conditional on $\{\go_1\fto x_k\}$; furthermore both are independent given the independent events $\{\go_1\fto x_{k'}\}\cap\{x_k\ffrom\come_n\}$. Hence, using invariance of $x_n$ by permutation of distances, 
\[\EE^{(n)}[f(x_n)\indic_{(\go_1\fto x_{k'})}\indic_{(x_{k'}\ffrom\come_n)}\indic_{(x_{k'}-x_1<x_n-x_{k'})}]=\EE^{(n)}[f(x_n)\indic_{(\go_1\fto x_{k})}\indic_{(x_{k}\ffrom\come_n)}\indic_{(x_{n}-x_k<x_k-x_{1})}].\]
Getting back to the previous summation, the comparison between distances simplifies, leaving independent conditions which by induction are independent of the widths:
\begin{align*}
\EE^{(n)}[f(x_n)\indic_{(0\ffrom\come_n)}\indic_{(\go_1\tto\stay)}]
	& =\frac p2\sum_{1<k<n}\EE^{(n)}[f(x_n)\indic_{(\go_1\fto x_k)}\indic_{(x_k\ffrom\come_n)}]\\
	& = \frac p2\sum_{1<k<n}\EE^{(n)}[f(x_n)]\PP(\go_1\fto x_k)\PP(x_k\ffrom\come_n) \\
	& = \EE^{(n)}[f(x_n)]\PP\bigl((0\ffrom\come_n)\wedge(\go_1\tto\stay)\bigr).
\end{align*}
Finally, 
\[\EE^{(n)}[f(x_n)\indic_{(\go_1\collide\come)}] = \sum_{1<k<n}\EE^{(n)}[f(x_n)\indic_{(\go_1\collide\come_k)}\indic_{(x_k\ffrom\come_n)}];\]
the last two conditions are independent, and by induction they don't affect the distribution of distances $x_k-x_1$ and $x_n-x_k$, so this case is handled as the first one. This altogether gives independence between $x_n$ and $\{0\ffrom\come_n\}$. 

It remains to consider $x_n$ and $\{\go_1\collide\come_n\}$. Using the same decomposition as in the proof of Proposition~\ref{pro:universality_A} (case of $\delta_n$) or in the alternative proof of Corollary~\ref{cor:xA} (page \pageref{proof_laplace}), we reduce to the independence between $x_j$ and $\{\go_1\tto\stay\}\cap\{0\ffrom\come_j\}$, and conclude as in the previous cases. 

b) Up to scaling, it is enough to prove the result for gamma distributions of scale parameter~$1$. 

We prove, by induction on the number $n$ of particles, that the result holds for a generalized model where some sites may be ``devoid of a particle'', which we formally handle by considering that the $n$ particles are separated by \emph{sums} of i.i.d.\ gamma interdistances, i.e.\ interdistances are again gamma distributed, with possibly different shape parameters (but same scale parameter $1$). In the following, when referring to a gamma distribution, it shall always be of scale parameter~$1$. 

The first nontrivial case is $n=3$, and only in the case of velocities $(+1,0,-1)$. Then the pairing depends on the comparison between $\ell_2$ and $\ell_3$, and we need to show that $\ell_2+\ell_3$ is independent of $\{\ell_2>\ell_3\}$ (remember $\ell_2$ and $\ell_3$ may have different distributions). This comes from the following classical property (see \cite[Section 4.11]{GS}): 
\begin{fact}
if $X$ and $Y$ are independent random variables with respective distributions $\Gamma(\alpha,1)$ and $\Gamma(\beta,1)$, then $X+Y$ is independent of $\Big(\frac X{X+Y},\frac Y{X+Y}\Big)$, hence in particular of $\{X<Y\}$, and has distribution $\Gamma(\alpha+\beta,1)$.
\end{fact}
Let us now assume $n\ge4$ and that the result holds for strictly fewer particles. Let $\alpha_1,\ldots,\alpha_n>0$ be given. We consider $n$ particles, and assume the interdistances $\ell_1,\ldots,\ell_n$ to have respective distributions $\Gamma(\alpha_1,1),\ldots,\Gamma(\alpha_n,1)$. Let a configuration $(v,\pi)$ be given. 

\textit{Case 1.} First consider the case when, in the configuration, there are indices $k<l$, different from $(1,n)$, such that $\go_k\collide\come_l$. Then, thanks to the induction applied to the strict subinterval $\iinter{k,l}$, conditional on the configuration $(v,\pi)$, the distance $x_l-x_k$ has same distribution as unconditionally (i.e.\ a gamma distribution, as a sum of independent gamma variables), and in particular same distribution as with particles $\bullet_k,\ldots,\bullet_l$ removed. Since such a pyramid shaped subconfiguration is independent of the configuration outside this subinterval (indeed no collision with particles outside $\iinter{k,l}$ is possible), we further conclude that the total width~$x_n$, conditioned on $(v,\pi)$, has same distribution as with particles $\bullet_k,\ldots,\bullet_l$ removed (including, from configuration $\pi$). This reduces to a strictly smaller number of particles, enabling to use again the induction to conclude. 

It remains to consider the cases when either $\go_1\collide\come_n$ or all collisions are of the type $\go\tto\stay$ or symmetrically. 

\textit{Case 2.}
Assume that $\go_1\collide\come_n$. We further consider two subcases.

\textit{Case 2.a.} If the configuration contains indices $k,l$ with $l>k+2$ and $\stay_k\from\come_l$ or $\go_k\tto\stay_l$, we have by the induction applied to the interval $\iinter{k,l}$ that $x_l-x_k$ is gamma distributed given the configuration. Since the configuration outside $\iinter{k+1,l-1}$ is independent of the configuration in $\iinter{k+1,l-1}$, given $\stay_k\from\come_l$ (or symmetrically), we conclude that the total width $x_n$ is distributed as with $\bullet_{k+1},\ldots,\bullet_{l-1}$ removed. This enables to use the induction and conclude in this subcase. 

\textit{Case 2.b.} Otherwise, the pairing in $\iinter{2,n-1}$ must be between neighbors: 
\begin{equation}\label{eq:pairing}
\pi=\begin{pmatrix}n & 3 & 2 & 5 & 4 & \cdots & n-1 & n-2 & 1\end{pmatrix},
\end{equation}
and each pair $(2k,2k+1)$ has either velocities $(+1,0)$ or $(0,-1)$, for $2\le 2k\le n-2$. It suffices to show that the law of $x_n$ given any other pairing, and given these same velocities, is unconditioned. Summing over all pairings (multiplied by their probabilities) indeed reduces to the law of $x_n$ given the velocities, which is nothing but the law of $x_n$ since the two are independent of each other. The only possible pairings compatible with these velocities, besides the previous one~\eqref{eq:pairing}, are of the following type (if any): either for some even index $2\le k\le n-2$, such that $v_k=0$, $\go_1\tto\stay_k$, while $\come_{k+1}$ does not collide, or symmetrically $\stay_{k'}\from\come_n$ while $\go_{k'-1}$ does not collide, for some $k'$ such that $v_{k'}=0$, or both happen, while other neighboring pairs are preserved (see also Figure~\ref{fig:pairing}). We notice that the realization of this configuration on $\iinter{1,k+1}$ and on $\iinter{k+2,n}$ (or symmetrically with $k'$) are independent, so that we can apply induction on each of these strict subintervals to show that their width are unaffected by conditioning on the subconfiguration. This concludes this subcase. 

\begin{figure}[h!t]
\begin{minipage}{.5\textwidth}
\includegraphics[width=\textwidth,page=1]{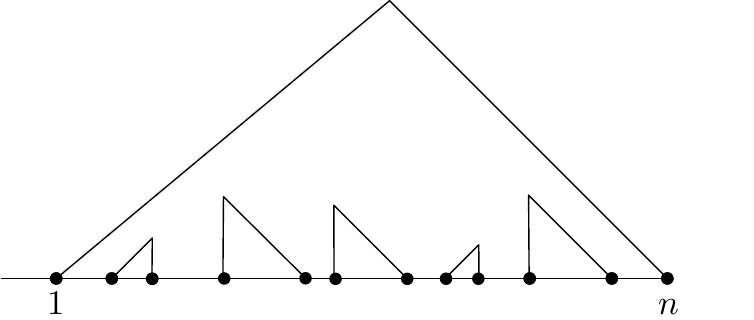}
\end{minipage}%
\begin{minipage}{.5\textwidth}
\includegraphics[width=\textwidth,page=2]{pairings_gamma.pdf}\\
\includegraphics[width=\textwidth,page=3]{pairings_gamma.pdf}
\end{minipage}
\caption{Pairing $\pi$ from~\eqref{eq:pairing} (left), and general form of the other pairings on the same velocities, up to left-right symmetry (right)}\label{fig:pairing}
\end{figure}

\textit{Case 3.} Finally, let us treat the case of configurations without any collision of the type $\go\collide\come$. Similarly to Case 2.a, we may apply induction to any configuration that has ``nested collisions'', i.e.\ $\bullet_i\sim\bullet_j$ for some $j\ge i+2$. We may therefore assume that collisions are between neighbors. Some particles may also not collide at all. However, if some particle $\bullet_i$ with $1<i<n$ does not collide, then the conditions on the configuration on the left and on the right of this particle (including the particle with the side where it is heading to if $v_i=\pm1$, and with neither if $v_i=0$) are independent, enabling to use induction as in the end of Case 2.b. Also, if $\bullet_1$ or $\bullet_n$ is surviving with velocity $0$, or with velocity $-1$ or $+1$ respectively, then the condition only leans on the other particles, enabling induction again. All in all, either all particles collide, in which case the pairing is necessarily between neighbors hence doesn't correlate with $x_n$, or only $\go_1$ or $\come_n$ survives. This last subcase is dealt with exactly as in Case 2.b, namely by treating the case of any other pairing on the same velocities, which describes as in Case 2.b and brings up conditions that split into independent conditions on subconfigurations, enabling to use induction and finally conclude. 
\end{proof}

\section{Variation of \texorpdfstring{$A$}{A} with respect to \texorpdfstring{$p$}{p}}~\label{sec:monotonicity_A}

In this section we consider, in the setting of independent random lengths, how the (universal) distribution of $A$ (i.e.\ the index of the leftmost particle that crosses 0) varies with the density $p$ of static particles. Note that $A$ is not monotonic under individual changes to the velocities of particles, and that merely reversing the direction of a single right-moving particle can even alter $A$ from finite to infinite. However, we conjecture that the law of $A$ is affected monotonically by changing $p$. As in Proposition~\ref{pro:universality_A}, let us denote, for $n\in\N$ and implicitly $p\in[0,1]$, 
\[p_n=\prb{A=n}=\PP(0\ffrom\come_n).\] 

We give three conjectures supported by computer-assisted computations for small values of~$n$. The first one states that in the supercritical region, each individual probability corresponding to a finite value of $A$ is decreasing in $p$: 
\begin{conjecture}\label{pn-dec}For each $n$, the function $p\mapsto p_n$ is monotonically decreasing on $[1/4,1]$.\end{conjecture}
This conjecture cannot be extended beyond this region: since $p_1=(1-p)/2$ is strictly decreasing on $[0,1]$, $\prb{A=\infty}$ must be strictly increasing wherever the conjectured result holds; however it is constant, equal to $0$, on $[0,1/4]$.

We also conjecture that we have stochastic dominance between the laws of $A$ for any two values of $p$, even in the subcritical region:
\begin{conjecture}\label{dominance}For each $n$, the function $p\mapsto \prb{A\leq n}$ is monotonically decreasing on $[0,1]$.\end{conjecture}

Finally, consider $\prb{A=n\mid\come_n}=\frac2{1-p}p_n$, which may equivalently be thought of as the probability that the first $n-1$ particles all annihilate one another in such a way that none of them would be in the path of a left-moving particle starting at $x_n$. We conjecture that this probability peaks at the same value $1/4$ for any $n$ (note that this critical value only appeared in the context of an infinite system so far):
\begin{conjecture}\label{max-cond}For each $n$, the function $p\mapsto\frac2{1-p}p_n$ is maximized at $p=1/4$.\end{conjecture}

We now give partial results to support these conjectures. First, Conjecture \ref{pn-dec} holds for simple reasons on a restricted range of values of $p$.
\begin{proposition}For each $n$, the function $p\mapsto p_n$ is monotonically decreasing on $[\frac12-\frac1{2n},1]$.
\end{proposition}
\begin{proof}Fix a particular law of interdistances $m$; recall that this does not affect $p_n$. The event $\{A=n\}$ may only occur when at most $\frac{n-1}2$ particles among the first $n$ are static, hence
\[p_n=\sum_{0\leq k\leq\frac{n-1}{2}}\sum_{\mathbf w\in\mathcal V_k}\prb{A=n\mid (v_1,\ldots,v_n)=\mathbf w}p^{k}(1-p)^{n-k}2^{k-n},\]
where $\mathcal V_k$ denotes the set of velocities of the first $n$ particles among which exactly $k$ are $0$. 
For fixed $k$, $p^k(1-p)^{n-k}$ is monotonically decreasing on $[k/n,1]$, and so every term in the above sum is monotonically decreasing on the required interval.
\end{proof}
We also observe that the function $p_n$ is decreasing around the critical value $1/4$. 
\begin{proposition}For each $n$, at $p=1/4$ we have $p'_n=-\frac43p_n$, where $p'_n=\frac{{\mathrm d}p_n}{\mathrm dp}$ \end{proposition}
\begin{proof}
Recall from~\eqref{eqn:rec_pn} that $p_1=(1-p)/2$ and, for all $n\ge2$, 
\[p_n=\Bigl(p+\frac12\Bigr)\sum_{\substack{k_1+k_2\\=n-1}}p_{k_1}p_{k_2}-\frac p2\sum_{\substack{k_1+k_2+k_3+k_4\\=n-1}}p_{k_1}p_{k_2}p_{k_3}p_{k_4}.\]
We prove the claimed statement by induction on $n$; it is easy to verify for $n=1$. Suppose it is true for all values less than $n$. Note that
\begin{align*}
p'_n
&=\sum_{\substack{k_1+k_2\\=n-1}}\biggl(p_{k_1}p_{k_2}+\Bigl(p+\frac12\Bigr)(p'_{k_1}p_{k_2}+p_{k_1}p'_{k_2})\biggr)\\
&\mathrel{\phantom{=}}-\sum_{\substack{k_1+k_2+k_3+k_4\\=n-1}}\Bigl(\frac12p_{k_1}p_{k_2}p_{k_3}p_{k_4}+\frac p2(p'_{k_1}p_{k_2}p_{k_3}p_{k_4}+\cdots+p_{k_1}p_{k_2}p_{k_3}p'_{k_4})\Bigr).
\end{align*}
Evaluating at $p=1/4$, assuming the induction hypothesis, gives
\begin{align*}p'_n&=\sum_{\substack{k_1+k_2\\=n-1}}\Bigl(1-2\times\frac34\times\frac43\Bigr)p_{k_1}p_{k_2}-\sum_{\substack{k_1+k_2+k_3+k_4\\=n-1}}\Bigl(\frac12-4\times\frac18\times\frac43\Bigr)p_{k_1}p_{k_2}p_{k_3}p_{k_4}\\
&=-\frac43\biggl(\frac34\sum_{\substack{k_1+k_2\\=n-1}}p_{k_1}p_{k_2}-\frac 18\sum_{\substack{k_1+k_2+k_3+k_4\\=n-1}}p_{k_1}p_{k_2}p_{k_3}p_{k_4}\biggr)=-\frac43p_n,\end{align*}
as required.
\end{proof}
This explicit logarithmic derivative in fact also gives support to Conjecture \ref{max-cond}, since it equivalently states that the derivative in $p$ of $\frac2{1-p}p_n$ is $0$ at $p=1/4$. 

Finally, we can give some additional support to Conjecture \ref{dominance} by exactly evaluating the (right-hand) derivative of $p_n$ at $0$. We may assume $n=2m+1$ is odd and at least $3$, since for $n$ even $p_n\equiv0$ and $p_1=(1-p)/2$. We prove the following.
\begin{theorem}\label{0deriv}The right-hand derivative of $p_{2m+1}$ at $p=0$ is $\frac{8m-5}{(m+1)(2m+4)}\binom{2m}m2^{-2m-1}$.\end{theorem}
Since $\prb{A\leq n}=\prb{A=\infty}-\sum_{k>n}p_k$, and $\prb{A=\infty}$ is constant on the subcritical region, we immediately obtain the following consequence.
\begin{corollary}The right-hand derivative of $\prb{A\leq n}$ at $0$ is negative for every $n\geq1$.\end{corollary}
\begin{proof}[Proof of Theorem \ref{0deriv}]
Let $n=2m+1$. Let $S$ denote the number of static particles among the first $n$, and observe that the law of $(v_1,\ldots,v_n)$ given $S$ does not depend on $p$. We have, as $p\to0$, 
\begin{align*}
p_n = \PP(A=n)
	& =\sum_{k=0}^n\PP(A=n\s S=k)\binom{n}{k}p^k(1-p)^{n-k}\\
	& =\PP(A=n\s S=0)(1-np)+\PP(A=n\s S=1)np+O(p^2),
\end{align*}
and the probabilities on the right-hand side do not depend on $p$ by the previous remark, hence the derivative of $p_n$ at $0$ is given by $n(\prb{A=n\mid S=1}-\prb{A=n\mid S=0})$.

First, we condition on $S=0$. In this case, the event $\{A=n\}$ means that $v_n=-1$, which has conditional probability $1/2$, and that the first $2m$ particles mutually annihilate. An arrangement of $2m$ particles which mutually annihilate corresponds precisely to an expression of $2m$ correctly-matched parentheses, and so the number of  such arrangements is equal to $C_m$, the $m$th Catalan number, which is given by $C_m=\frac1{m+1}\binom{2m}{m}$. Thus, 
\[\PP(A=n\s S=0)=\frac12\cdot\frac{C_m}{2^{2m}}.\]

Next, we turn to the case $S=1$. For simplicity we consider the case of constant interdistances (with triple collisions resolved at random); by universality, this is sufficient. In this case, $A=n$ means for the last particle to be left-moving (which occurs with conditional probability $\frac{m}{2m+1}$) and for the remaining $2m$ particles, of which one is static, to mutually annihilate in such a way that they do not interfere with the last particle; this includes cases where the last particle survives a triple collision. Thus, 
\[\PP(A=n\s S=1)=\frac m{2m+1}\triright_m,\]
where we write $\triright_m$ for the number of such arrangements of $2m$ velocities: we may indeed think of it as a requirement for $2m$ particles (one of which is static) to annihilate and have space-time trajectories contained inside the triangle described by the trajectories of a static particle at $0$ and a left-moving particle at $2m+1$ -- here we count a collision happening exactly on the right-hand side of this triangle as ``inside'' only if the spin of the static particle is $-1$. Note that $\triright_m$ could be a half integer as some arrangements require a particular spin hence count $1/2$. 

By symmetry, $\triright_m=\trileft_m$, where the latter is the number of arrangements which mutually annihilate inside the reflection of the previous triangle. We consider two other similar quantities: write $\square_m$ for the number of arrangements of $2m$ velocities, of which one is static, which mutually annihilate, and $\triboth_m$ for the number of such arrangements which mutually annihilate inside the space-time triangle described by the trajectories of a right-moving particle at $0$ and a left-moving particle at $2m+1$. If a configuration is counted in $\square_m$ but not in $\triboth_m$, then it counts inside exactly one of $\trileft_m$ or $\triright_m$, otherwise it is in both, thus $\trileft_m+\triright_m=(\square_m-\triboth_m)+2\triboth_m$, i.e.\ $\square_m+\triboth_m=2\,\triright_m$.

We claim that $\triboth_m=mC_m$. This is equivalent to the claim that if we take a random set of $2m+2$ moving particles, conditioned on the first and last colliding (this leaves $C_m$ uniform choices), and make a random internal particle static ($2m$ choices), then with probability $1/2$ the first and last still collide together.

We prove this by induction on $m$. The case $m=1$ is straightforward; consider $m\ge2$ and assume the property true in the previous cases. Suppose the particle chosen to become static is not in the ``skyline'' of the $2m$ internal particles, i.e.\ it is between two colliding particles other than the first and last. Then the probability that these two particles still collide is $1/2$ by induction. If they do, the outer particles are unaffected, but if not then one of them is released to collide with an outer particle. Thus it suffices to prove the claim for a particle chosen in the skyline and, by symmetry, we may assume this particle is right-moving.

Consider the Dyck paths corresponding to configurations of internal particles, with a step $x$ from $0$ to $1$ (i.e.\ corresponding to a right-moving particle in the skyline) marked. Let $y$ be the next step from $1$ to $0$, $\mathbf{a}$ be the subpath before the marked step, and $\mathbf{b}$ be the subpath of steps strictly between $x$ and $y$. Making the particle corresponding to $x$ static will cause a collision with one of the external particles if and only if $\lvert\mathbf a\rvert+1<\lvert\mathbf b\rvert+1$ (or with probability $1/2$ if they are equal), since these are the distances to the two particles which could collide with $x$. Swapping the subpaths $\mathbf a$ and $\mathbf b$ gives another Dyck path, so this bijective transformation keeps the particle corresponding to $x$ in the skyline and right-moving; and it maps any configuration where $x$ would be colliding with an external particle if made static, to one where its corresponding particle would not, and vice-versa, so this proves the claim.

We apply a similar argument to calculate $\square_m$: starting from a totally annihilating configuration of $2m$ moving particles, a random one is made static. By the previous claim, if this particle is not in the skyline, the change has chance $1/2$ of preserving total annihilation. However, making a particle in the skyline static always preserves total annihilation. Thus if a random configuration of $2m$ moving particles which mutually annihilate is modified by making a random particle $x$ static, the probability that all particles still annihilate is $\frac12\prb{\text{$x$ not in skyline}}+\prb{\text{$x$ in skyline}}=\frac12+\frac12\EE[W]$, where $W$ is the number of particles among $\iinter{1,2m}$ that are in the skyline. 

Note that $W=2V-2$, where $V$ is the number of visits to $0$ by the Dyck path (including the start and end of the path). The number of arrangements which visit $0$ after $2k$ steps is $C_kC_{m-k}$ for each $k\in\iinter{0,m}$, and so
\[\EE[V]=\sum_{k=0}^m\frac{C_kC_{m-k}}{C_n}=\frac{C_{m+1}}{C_m}=\frac{4m+2}{m+2};\]
it follows that $\EE[W]=\frac{6m}{m+2}$,  
giving finally
\[\square_m=2mC_m\Bigl(\frac12+\frac12\cdot\frac{3}{m(m+2)}\Bigr)=\frac{m(m+5)}{m+2}C_m.\]
Consequently
\[\triright_m=\frac{\square_m+\triboth_m}2=\frac{m(2m+7)}{2m+4},\]
hence
\[\prb{A=2m+1\mid S=1}=\frac{2m(2m+7)}{(2m+1)(2m+4)}C_m2^{-2m}.\]  
The result follows by gathering the previous computations.
\end{proof}

\section*{Acknowledgements}
J.H.\ was supported by the European Research Council (ERC) under the European Union's Horizon 2020 research and innovation programme (grant agreement no.\ 639046). L.T.\ was supported by the French ANR project MALIN (ANR-16-CE93-000). 


\end{document}